\documentclass[pdf, 10pt]{amsart}

\usepackage{amsthm, amsmath, amssymb}
\usepackage[OT1]{fontenc}
\usepackage[colorlinks,citecolor=blue,urlcolor=blue]{hyperref}
\usepackage{color}

\newtheorem{thm}[equation]{Theorem}
\newtheorem{cor}[equation]{Corollary}

\newtheorem{defn}[equation]{Definition}

\newtheorem{lem}[equation]{Lemma}

\newtheorem{rem}[equation]{Remark}

\def\R{{\mathbb{R}}}

\def\N{{\mathbb{N}}}
\def\Z{{\mathbb{Z}}}

\newcommand{\supp}{\mbox{supp}}
\renewcommand{\a}{\alpha}
\renewcommand{\b}{\beta}
\newcommand{\g}{\gamma}

\renewcommand{\d}{\delta}

\newcommand{\la}{\lambda}

\newcommand{\eps}{\epsilon}
\newcommand{\e}{\varepsilon}

\renewcommand{\t}{\tau}

\newcommand{\te}{\theta}
\newcommand{\s}{\sigma}

\newcommand{\vp}{\varphi}

\newcommand{\8}{\infty}

\newcommand{\vt}{\vartheta}
\newcommand{\vr}{\varrho}

 \numberwithin{equation}{section}

\begin{document}

\title[] {Weak type $(1, 1)$ inequalities for discrete rough maximal functions}
\author[M.Mirek]
{Mariusz Mirek}
\address{M.Mirek \\
Universit\"{a}t Bonn \\
Mathematical Institute\\
Endenicher Allee 60\\
D--53115 Bonn \\
Germany \&\\
Uniwersytet Wroc{\l}awski\\
Mathematical Institute \\
Plac Grunwaldzki 2/4\\
 50--384 Wroc{\l}aw\\
 Poland}
 \email{mirek@math.uni.wroc.pl}

\thanks{
The author was supported by NCN grant DEC--2012/05/D/ST1/00053}
\maketitle

\begin{abstract}
The aim of this paper is to show that the discrete maximal function
$$\mathcal{M}_{h}f(x)=\sup_{N\in\mathbb{N}}\frac{1}{|\mathbf{N}_{h}\cap[1, N]|}\Big|\sum_{n\in
\mathbf{N}_{h}\cap[1, N]}f(x-n)\Big|,\ \ \mbox{for \ $x\in\mathbb{Z}$},$$
is of weak type $(1, 1)$, where $\mathbf{N}_{h}=\{n\in\mathbb{N}:
\exists_{m\in\mathbb{N}}\ n=\lfloor h(m)\rfloor\}$  for an appropriate function $h$.
As a consequence we also obtain pointwise ergodic theorem along the set $\mathbf{N}_{h}$.
\end{abstract}

\section{Introduction and statement of results}
Recently, Urban and Zienkiewicz in \cite{UZ} proved that the maximal function
\begin{align}\label{maximal:1}
  \mathcal{M}f(x)=\sup_{N\in\N}\frac{1}{N}\Big|\sum_{n=1}^Nf(x-a_n)\Big|,
\ \ \mbox{for \ $x\in\Z$},
\end{align}
is of weak type $(1, 1)$, for $a_n=\lfloor n^c\rfloor$ where $1<c<1.001$ giving a negative answer for Rosenblatt--Wierdl's conjecture -- for more details and the historical background we refer to \cite{RW} page 74. Not long afterwards, LaVictoire \cite{LaV} and Christ \cite{C1} provided some new examples of sequences $(a_n)_{n\in\N}$ having Banach density $0$, for which maximal function $\mathcal{M}f$ is of weak type of $(1, 1)$.

The main aim of this article is to study maximal functions
\begin{align}\label{max:1}
  \mathcal{M}_{h}f(x)=\sup_{N\in\N}\frac{1}{|\mathbf{N}_{h}\cap[1, N]|}\Big|\sum_{n\in
\mathbf{N}_{h}\cap[1, N]}f(x-n)\Big|,\ \ \mbox{for \ $x\in\Z$},
\end{align}
defined along subsets of integers $\mathbf{N}_{h}$ of the form
\begin{align}\label{def:N}
  \mathbf{N}_{h}=\{n\in\N:
\exists_{m\in\N}\ n=\lfloor h(m)\rfloor\},
\end{align}
where $h$ is an appropriate function, see Definition \ref{def:1}. We are going to consider such functions $h$ for which $\mathcal{M}_{h}f$ is of weak type $(1, 1)$ -- see Theorem \ref{thm:1max} below. Our motivation to study maximal functions \eqref{max:1} for arithmetic sets defined in \eqref{def:N} is that: on the one hand, we were inspired by the series of papers of Bourgain \cite{B1},
\cite{B2} and \cite{B3}  where he proved $\ell^p(\Z)$ -- boundedness $(p>1)$ of ergodic averages modeled on integer valued polynomials and the recent results of Buczolich and Mauldin \cite{BM} and LaVictoire \cite{LaV1}. They showed that the pointwise convergence of ergodic averages along $p(n)=n^k$ for $k\ge2$ fails on $L^1$. On the other hand, we did not know (apart from the example given in \cite{UZ}) any considerable examples of sequences (given by a concrete formula) for which $\mathcal{M}f$ is of weak type of $(1, 1)$. Similar problems were studied in \cite{BKQW} in the context of $L^p$ -- boundedness $(p>1)$ of ergodic averaging operators, but the case of $p=1$ remained unresolved until the results of \cite{B}, \cite{UZ}, \cite{LaV} and \cite{C1}.
Here we will make the first attempt at characterizing a class of functions $h$ for which maximal function in \eqref{max:1} is of weak type $(1, 1)$.

 The sets  $\mathbf{N}_{h}$,  considered as subsets of the set of prime numbers $\mathbf{P}$, have great importance in analytic number theory. Namely, in 1953  Piatetski--Shapiro  established an asymptotic formula for
$$\mathbf{P}_{\g}=\{p\in\mathbf{P}: \exists_{n\in\N}\
p=\lfloor n^{1/{\g}}\rfloor\}=\mathbf{N}_{x^{1/{\g}}}\cap\mathbf{P},$$
of fixed type $\g<1$ ($\g$ is sufficiently close to $1$). More precisely,  it was shown in \cite{PS} that
$$|\mathbf{P}_{\g}\cap[1, x]|\sim\frac{x^{\g}}{\log x}  \ \ \mbox{as \ $x\to\8$},$$ for every $\g\in(11/12, 1)$.
Recently, the author \cite{M} proved $\ell^p(\Z)$ -- boundedness $(p>1)$  of maximal functions modeled on  subsets of primes of the form $\mathbf{N}_{h}\cap\mathbf{P}$, for $h$ as in Definition \ref{def:1}. In \cite{M} we  have also obtained related pointwise ergodic theorems and showed that the ternary Goldbach problem has a solution in the primes belonging to $\mathbf{N}_{h}\cap\mathbf{P}$. On the other hand, in \cite{M1}, we proved a counterpart of Roth theorem for the Piatetski--Shapiro primes.

Throughout the paper we will use the convention that $C > 0$ stands for a large positive constant whose value may change from line to line. For two quantities $A>0$ and $B>0$ we say that $A\lesssim B$ ($A\gtrsim B$) if there exists an absolute constant $C>0$ such that $A\le CB$ ($A\ge CB$). If $A\lesssim B$ and $A\gtrsim B$ hold simultaneously then we will shortly write that $A\simeq B$. We will also write $A\lesssim_{\d} B$ ($A\gtrsim_{\d} B$) to indicate that the constant $C>0$ depends on some $\d>0$.

\begin{defn}\label{def:1}
Let $c\in(1, 2)$ and $\mathcal{F}_c$ be the family of all functions $h:[x_0, \8)\mapsto [1, \8)$ (for some $x_0\ge1$)  satisfying
\begin{enumerate}
\item[(i)] $h\in \mathcal{C}^3([x_0, \8))$ and
$$h'(x)>0,\ \ \ \ h''(x)>0, \ \ \mbox{for every \  $x\ge x_0$.}$$
\item[(ii)] There exists a real valued function $\vartheta\in\mathcal{C}^3([x_0, \8))$ and a constant $C_h>0$ such that
\begin{align}\label{eq1}
  h(x)=C_hx^c\ell_h(x), \ \ \mbox{where}\ \ \ell_h(x)=e^{\int_{x_0}^x\frac{\vt(t)}{t}dt}, \ \ \mbox{for every \  $x\ge x_0$,}
\end{align}
and
\begin{align}\label{eq2}
  \lim_{x\to\8}\vartheta(x)=0,\ \ \lim_{x\to\8}x\vartheta'(x)=0,\ \ \lim_{x\to\8}x^2\vartheta''(x)=0\ \ \lim_{x\to\8}x^3\vartheta'''(x)=0.
\end{align}
\end{enumerate}
\end{defn}
Among the functions belonging to the family $\mathcal{F}_c$ are (up to multiplicative constant $C_h>0$)
\begin{align*}
  h_1(x)=x^c\log^Ax,\ \ h_2(x)=x^ce^{A\log^Bx},\ \ h_3(x)=x^cl_m(x),
\end{align*}
where $c\in(1, 2)$, $A\in\R$, $B\in(0, 1)$, $C>0$, $l_1(x)=\log x$ and $l_{m+1}(x)=\log(l_m(x))$, for $m\in\N$.

From now on we will focus our attention  on  subsets of integers $\mathbf{N}_h$ defined in \eqref{def:N} with $h\in\mathcal{F}_c$. Let $\d_n(x)$ stands for Dirac's delta, i.e. $\d_n(x)=1$ if $x=n$, and $\d_n(x)=0$ otherwise. Our main result is the following.
\begin{thm}\label{thm:1max}
Assume that $c\in(1, 30/29)$ and $h\in\mathcal{F}_c$. Let $\eta\in\mathcal{C}^{\8}(\R)$ be a smooth cut--off function supported in $(1/2, 4)$ such that $\eta(x)=1$ for $x\in[1, 2]$ and $0\le\eta(x)\le 1$ for $x\in\R$. Define a maximal function
\begin{align}\label{thm:1max1}
  \mathcal{M}_{h}f(x)=\sup_{N\in\N}|K_{h, N}*f(x)|\ \ \mbox{for  \  $x\in\Z$,}
\end{align}
corresponding with the kernel
\begin{align}\label{thm:1max2}
  K_{h, N}(x)=\frac{1}{|\mathbf{N}_h\cap[1, N]|}\sum_{n\in\mathbf{N}_h\cap[1, N]}\d_n(x)\eta\left(\frac{n}{N}\right) \ \ \mbox{for  \  $x\in\Z$.}
\end{align}
Then
\begin{align}\label{thm:1max3}
  \|\mathcal{M}_{h}f\|_{\ell^{1, \8}(\Z)}\lesssim\|f\|_{\ell^1(\Z)},
\end{align}
for every $f\in\ell^1(\Z)$. In particular $\mathcal{M}_{h}f$ is bounded on $\ell^p(\Z)$ for every $f\in\ell^p(\Z)$ and $p>1$.
\end{thm}
The boundedness of \eqref{thm:1max1} implies the boundedness of \eqref{max:1} (possibly with a different constant) and vice versa. Therefore, it will cause no confusion if we use the same letter $\mathcal{M}_{h}f$ in the definitions \eqref{max:1} and \eqref{thm:1max1}.
The proof of Theorem \ref{thm:1max} (see Section \ref{sectmax}) will be based on the concepts of \cite{UZ}. In \cite{UZ} the authors used a subtle version of Calder\'{o}n--Zygmund decomposition, which was pioneered by Fefferman \cite{F} and later on  developed by Christ \cite{C}, to study maximal functions. Fefferman's ideas turned out to be applicable to the discrete settings as it was shown in \cite{UZ}, and recently also in \cite{LaV} and \cite{C1}.
 Heuristically speaking, the weak type
$(1, 1)$ bound of $\mathcal{M}_hf$ is obtained by considering the recalcitrant part of the Calder\'{o}n--Zygmund decomposition in $\ell^2$ (see Lemma \ref{lem:4} and Theorem \ref{thm:3CZ} in Section \ref{sectmax}), using the fact that
$\langle K_{h, N}*f, K_{h, N}*g\rangle=\langle K_{h, N}*\widetilde{K}_{h, N}*f, g\rangle$, and decomposing $K_{h, N}*\widetilde{K}_{h, N}(x)$\footnote{where $\widetilde{K}_{h, N}(x)=K_{h, N}(-x)$}
into several manageable pieces (a delta mass at $0$, a slowly varying function $G_{N}(x)$, and a small error term $E_{N}(x)$, see Section \ref{section4}) obtained by special Van der Corput estimates, (we refer to Section \ref{section3}).

 As we mentioned before our motivations to study such maximal functions are derived in part by scant knowledge of the structure of  functions $h$ for which $\mathcal{M}_{h}f$ is of weak type $(1, 1)$. The family $\mathcal{F}_c$ was studied in \cite{M} to generate various thin subsets of primes in the context
of pointwise ergodic theorems and it turned out to be a good candidate to improve
qualitatively  theorem from \cite{UZ}. On the other hand the family $\mathcal{F}_c$ gives rise to renew the discussion initiated in \cite{BKQW} and sheds some new light on $L^1$ -- pointwise ergodic theorems which have not been brought up there.
It is worth pointing out that the complexity of the family $\mathcal{F}_c$  causes
some obstructions which did not occur in \cite{UZ}. Namely, we had to completely change   the method of approximation of the kernel $K_{h, N}*\widetilde{K}_{h, N}(x)$ compared to the method form \cite{UZ} and this is the novelty of this paper (see Section \ref{section3} and Section \ref{section4}). Their approach is inadequate here since it leads us to study exponential sums with a complicated form of a phase function, and loosely speaking this is the reason why we prefer to consider $\mathbf{N}_h\cap[1, N]$ in \eqref{max:1} instead of $\{\lfloor h(m)\rfloor: m\in[1, N]\}$.

Now we have to emphasize that our method does not settle the case when $c=1$. It would be nice to know, for instance, if $\mathcal{M}_hf$ is of weak type $(1, 1)$ for $h(x)=x\log x$. We hope to return this matter at a future time.
Although the argument as stated works only for $1 < c < 30/29$, the
obstacles involved in getting a similar result for $1 < c < 2$ pale in comparison to the obstacles for $c > 2$,
since at that point $K_{h, N}*\widetilde{K}_{h, N}(x)$ no longer has any useful properties. Nevertheless, LaVictoire \cite{LaV} and Christ \cite{C1} provided a certain wide class of sequences for which $\mathcal{M}f$ from \eqref{maximal:1} is of weak type $(1, 1)$.

 If it comes to $\ell^p(\Z)$ -- boundedness of $\mathcal{M}_hf$ for $p>1$, one can conclude, thanks to Lemma \ref{lemform1}, that it holds for all $h\in\mathcal{F}_c$ provided that $c\in[1, 4/3)$. However, if $c=1$ then the conditions in \eqref{eq2} from Definition \ref{def:1} must be modified in the following way.
\begin{rem}\label{rem:1}
If $c=1$, then we additionally assume that $\vt(x)$ is positive, decreasing and  for every $\varepsilon>0$
  \begin{align}\label{eq3}
    \frac{1}{\vt(x)}\lesssim_{\varepsilon}x^{\varepsilon}, \ \ \mbox{and} \ \ \lim_{x\to\8}\frac{x}{h(x)}=0.
  \end{align}
  Furthermore,
  \begin{align}\label{eq4}
  \lim_{x\to\8}\vartheta(x)=0,\ \ \lim_{x\to\8}\frac{x\vartheta'(x)}{\vt(x)}=0,\ \ \lim_{x\to\8}\frac{x^2\vartheta''(x)}{\vt(x)}=0,\ \ \lim_{x\to\8}\frac{x^3\vartheta'''(x)}{\vt(x)}=0.
\end{align}
\end{rem}

On the one hand, our approach supplies one more different method to the techniques developed in  \cite{BKQW} which permits us to treat with $L^p$ -- boundedness for ergodic averages. On the other hand, it is worth noting that some of the $L^p$ results for $p > 1$ are new, as there are some $h\in\mathcal{F}_c$ which do not belong to any Hardy field and are thus not covered by the results of \cite{BKQW}.

Theorem \ref{thm:1max} is the main ingredient in the following.
\begin{thm}\label{ergthm}
Assume that $c\in(1, 30/29)$ and $h\in\mathcal{F}_c$. Let $(X, \mathcal{B}(X), \mu, T)$ be a dynamical system, where $\mu$ is a $\sigma$--finite measure and $T$ is an invertible and measure preserving transformation on $X$. Then for every $f\in L^p(X, \mu)$ where $p\ge1$, the ergodic averages
\begin{align}\label{ergh}
  A_{h, N}f(x)=\frac{1}{|\mathbf{N}_h\cap[1, N]|}\sum_{n\in\mathbf{N}_h\cap[1, N]}f(T^nx) \ \ \mbox{for  \  $x\in X$},
\end{align}
converges $\mu$--almost everywhere on $X$.
\end{thm}

The paper is organized as follows. In Section \ref{section2} we give the necessary properties of function $h\in\mathcal{F}_c$ and its inverse $\vp$. In Section \ref{section3} we estimate some exponential sums which allow us to decompose the kernel $K_{h, N}*\widetilde{K}_{h, N}(x)$ in the penultimate section. Assuming momentarily Theorem \ref{thm:1max} (its proof has been postponed to the last section), we prove Theorem \ref{ergthm} in Section  \ref{secterg}.

Despite the fact that Theorem \ref{thm:1max} works only for $c\in(1, 30/29)$ we decided to formulate the results in Section \ref{section2} and Section \ref{section3} also for $c=1$ (see Remark \ref{rem:1}), mainly due to new examples of functions in the family $\mathcal{F}_1$.

\section*{Acknowledgements}
I would like to thank Jacek Zienkiewicz for years of discussions on related problems.

The author is grateful to the referee for careful reading
of the manuscript and useful remarks that led to the improvement
of the presentation.

\section{Basic properties of functions $h$ and $\vp$}\label{section2}
In this section we gather all necessary properties of function $h\in\mathcal{F}_c$ and its inverse $\vp$ and we follow the notation used in Section 2 from \cite{M1}.
\begin{lem}\label{filem}
Assume that $c\in[1, 2)$ and $h\in\mathcal{F}_c$. Then for every $i=1, 2, 3$ there exists a function $\vt_i:[x_0, \8)\mapsto\R$ such that
 \begin{align}\label{heq}
   xh^{(i)}(x)=h^{(i-1)}(x)(\a_i+\vt_i(x)),\ \ \mbox{for every \ $x\ge x_0$,}
 \end{align}
 where $\a_i=c-i+1$, $\vt_1(x)=\vt(x)$,
 \begin{align}\label{heq1}
   \vt_i(x)=\vt_{i-1}(x)+\frac{x\vt_{i-1}'(x)}{\a_{i-1}+\vt_{i-1}(x)},\ \ \mbox{for $i=2, 3$}\ \ \mbox{and}\ \ \lim_{x\to\8}\vartheta_i(x)=0,\ \ \mbox{for $i=1, 2, 3$}.
 \end{align}
 If $c=1$, then there exist constants $0<c_1\le c_2$ and a function $\vr:[x_0, \8)\mapsto[c_1, c_2]$, such that
  \begin{align}\label{vt2c1}
   \vt_2(x)=\vt(x)\vr(x),\ \ \mbox{for every \ $x\ge x_0$ \ and } \lim_{x\to\8}\frac{x\vt_2'(x)}{\vt_2(x)}=0.
 \end{align}
 In particular \eqref{heq} with $i=2$ reduces to
 \begin{align}\label{heqc1}
   xh''(x)=h'(x)\vt(x)\vr(x),\ \ \mbox{for every \ $x\ge x_0$.}
 \end{align}
 The cases for $i=1, 3$ remain unchanged.
\end{lem}
\begin{proof}
For the proof we refer to \cite{M1}.
\end{proof}

\begin{lem}\label{formfunlem}
Assume that $c\in[1, 2)$, $h\in\mathcal{F}_c$, $\g=1/c$ and let $\vp:[h(x_0), \8)\mapsto[x_0, \8)$ be its inverse. Then there exists a function $\te:[h(x_0),\8)\mapsto\R$ such that $x\vp'(x)=\vp(x)(\g+\te(x))$ and
  \begin{align}\label{funfi}
  \vp(x)=x^{\g}\ell_{\vp}(x),\ \ \ \mbox{where}\ \ \ \ell_{\vp}(x)=e^{\int_{h(x_0)}^x\frac{\te(t)}{t}dt+D},
\end{align}
  for every $x\ge h(x_0)$, where $D=\log (x_0/h(x_0)^{\g})$ and $\lim_{x\to\8}\te(x)=0$. Moreover,
  \begin{align}\label{tetadef}
  \te(x)=\frac{1}{(c+\vartheta(\vp(x)))}-\g
  =-\frac{\vartheta(\vp(x))}{c(c+\vartheta(\vp(x)))}.
\end{align}
  Additionally, for every $\e>0$
  \begin{align}\label{slowhfi}
    \lim_{x\to\8}x^{-\e}L(x)=0,\ \ \ \mbox{and}\ \ \ \lim_{x\to\8}x^{\e}L(x)=\8,
  \end{align}
  where $L(x)=\ell_h(x)$ or $L(x)=\ell_{\vp}(x)$. In particular, for every $\e>0$
  \begin{align}\label{ratefi}
    x^{\g-\e}\lesssim_{\e}\vp(x),\ \ \ \mbox{and}\ \ \ \lim_{x\to\8}\frac{\vp(x)}{x}=0.
  \end{align}
  Finally, $x\mapsto x\vp(x)^{-\d}$ is increasing for every $\d< c$, (if $c=1$, even $\d\le1$ is allowed) and for every $x\ge h(x_0)$ we have
  \begin{align}\label{compfi}
    \vp(x)\simeq\vp(2x),\ \ \mbox{and}\ \ \vp'(x)\simeq\vp'(2x).
  \end{align}
\end{lem}
\begin{proof}
For the proof we refer to \cite{M1}.
\end{proof}
The next lemma will be very important in the sequel.
\begin{lem}\label{intlem}
Assume that $h\in\mathcal{F}_c$ and let $\vp:[h(x_0), \8)\mapsto[x_0, \8)$ be its inverse. Then
\begin{align}\label{intlemform}
  p\in\mathbf{N}_{h} \Longleftrightarrow\ \lfloor-\vp(p)\rfloor-\lfloor-\vp(p+1)\rfloor=1,
\end{align}
for all sufficiently large $p\in\mathbf{N}_{h}$.
\end{lem}
\begin{proof}
For the proof we refer to \cite{M1}.
\end{proof}

We finish this section by proving Lemma \ref{funlemfi}.

\begin{lem}\label{funlemfi}
Assume that $c\in[1, 2)$, $h\in\mathcal{F}_c$, $\g=1/c$ and let $\vp:[h(x_0), \8)\mapsto[x_0, \8)$ be its inverse. Then
 for every $i=1, 2, 3,$ there exists  a function $\theta_i:[h(x_0), \8)\mapsto\R$ such that
 \begin{align}\label{fiequat}
   x\vp^{(i)}(x)=\vp^{(i-1)}(x)(\b_i+\theta_i(x)), \ \ \mbox{for every \  $x\ge h(x_0)$,}
 \end{align}
  where $\b_i=\g-i+1$,  $\lim_{x\to\8}\te_i(x)=0$ and $\lim_{x\to\8}x\te_i'(x)=0$. If $c=1$, then there exists a positive function $\s:[h(x_0), \8)\mapsto(0, \8)$ and a function $\t:[h(x_0), \8)\mapsto \R$ such that \eqref{fiequat} with $i=2$ reduces to
\begin{align}\label{fiequat1}
  x\vp''(x)=\vp'(x)\s(x)\t(x),\ \ \mbox{for every \  $x\ge h(x_0)$ \ and } \lim_{x\to\8}\frac{x\te_2'(x)}{\te_2(x)}=0.
\end{align}
The cases for $i=1, 3$ remain unchanged. Moreover, $\s(x)$  is decreasing, $\lim_{x\to\8}\s(x)=0,$ $\s(2x)\simeq\s(x),$ and  $\s(x)^{-1}\lesssim_{\varepsilon}x^{\varepsilon},$
for every $\varepsilon>0$. Finally, there are constants $0<c_3\le c_4$ such that $c_3\le-\t(x)\le c_4$ for every $x\ge h(x_0)$.
\end{lem}
\begin{proof}
The proof is based on simple computations. However, for the convenience of the reader we shall present the details. In fact, \eqref{fiequat} for $i=1$ with $\te_1(x)=\te(x)$, has been shown in  Lemma \ref{formfunlem}. Arguing likewise in the proof of Lemma \ref{filem} we obtain \eqref{fiequat} for $i=2, 3$. More precisely,
  \begin{align}\label{te1}
    \te_1(x)=\te(x)=-\frac{\vartheta(\vp(x))}{c(c+\vartheta(\vp(x)))}=\frac{1}{c+\vartheta(\vp(x))}-\g,
  \end{align}
  and $\te_1'(x)=\te'(x)=\left(\frac{1}{c+\vartheta(\vp(x))}-\g\right)'
  =-\frac{\vartheta'(\vp(x))\vp'(x)}{(c+\vartheta(\vp(x)))^2}.$ Thus
  \begin{align}\label{te2}
  \te_2(x)&=\te(x)+\frac{x\te'(x)}{\g+\te(x)}=\frac{1}{c+\vartheta(\vp(x))}-\g-
  \frac{\vartheta'(\vp(x))\vp(x)}{(c+\vartheta(\vp(x)))^2}=\Theta_2(\vp(x)),
\end{align}
where $\Theta_2(x)=\frac{1}{c+\vartheta(x)}-\g-\frac{\vartheta'(x)x}{(c+\vartheta(x))^2}$ and $\te_2'(x)=\Theta_2'(\vp(x))\vp'(x)$ with
\begin{align*}
  \Theta_2'(x)=\left(\frac{1}{c+\vartheta(x)}-\g-
  \frac{\vartheta'(x)x}{(c+\vartheta(x))^2}\right)'
  =-\frac{(\vt''(x)x+2\vt'(x))(c+\vartheta(x))
  -2\vt'(x)^2x}{(c+\vartheta(x))^3},
\end{align*}
\begin{align*}
  x\te_2'(x)=\frac{\Theta_2'(\vp(x))\vp(x)}{c+\vt(\vp(x))}=-\frac{(\vt''(\vp(x))\vp(x)^2+2\vt'(\vp(x))\vp(x))(c+\vartheta(\vp(x)))
  -2\vt'(\vp(x))^2\vp(x)^2}{(c+\vartheta(\vp(x)))^4}.
\end{align*}
Finally,
\begin{align}\label{te3}
  \te_3(x)=\te(x)+\frac{x\te'(x)}{\g+\te(x)}+\frac{x\te_2'(x)}{\g-1+\te_2(x)}=\te_2(x)+\frac{x\te_2'(x)}{\g-1+\te_2(x)},
\end{align}
Therefore, $\te_3(x)$ can be rewritten as $\te_3(x)=\Theta_3(\vp(x))$ where
\begin{align*}
  \Theta_3(x)&=\frac{1}{c+\vartheta(x)}-\g-\frac{\vartheta'(x)x}{(c+\vartheta(x))^2}-
  \frac{\frac{(\vt''(x)x^2+2\vt'(x)x)(c+\vartheta(x))
  -2\vt'(x)^2x^2}{(c+\vartheta(x))^4}}{\frac{1}{c+\vartheta(x)}-1-\frac{\vartheta'(x)x}{(c+\vartheta(x))^2}}\\
  &=\Theta_2(x)-\frac{(\vt''(x)x^2+2\vt'(x)x)}{(c+\vartheta(x))^2-(c+\vartheta(x))^3-\vartheta'(x)x(c+\vartheta(x))}\\
  &+\frac{2\vt'(x)^2x^2}{(c+\vartheta(x))^3-(c+\vartheta(x))^4-\vartheta'(x)x(c+\vartheta(x))^2},
\end{align*}
and $\te_3'(x)=\Theta_3'(\vp(x))\vp'(x)$ where
\begin{multline}\label{Tet3}
   \Theta_3'(x)=\Theta_2'(x)-\frac{\vt'''(x)x^2+4\vt''(x)x+2\vt'(x)}{(c+\vartheta(x))^2(1-c-\vartheta(x))-\vartheta'(x)x(c+\vartheta(x))}\\
   +\frac{(\vt''(x)x^2+2\vt'(x)x)((c+\vartheta(x))\vt'(x)-3(c+\vartheta(x))^2\vt'(x)
   -\vt''(x)x(c+\vartheta(x))-\vt'(x)^2x)}{(((c+\vartheta(x))^2(1-c-\vartheta(x))-\vartheta'(x)x(c+\vartheta(x)))^2}\\
   +\frac{4\vt'(x)x(\vt''(x)x+\vt'(x))}{(c+\vartheta(x))^3(1-c-\vartheta(x))-\vartheta'(x)x(c+\vartheta(x))^2}\\
   -\frac{2\vt'(x)^2x^2(2(c+\vartheta(x))^2\vt'(x)-4(c+\vartheta(x))^3\vt'(x)
   -\vartheta''(x)x(c+\vartheta(x))^2-2\vartheta'(x)^2x(c+\vartheta(x)))}
   {((c+\vartheta(x))^3(1-c-\vartheta(x))-\vartheta'(x)x(c+\vartheta(x))^2)^2}.
\end{multline}
These computations and \eqref{eq2} yield $\lim_{x\to\8}\te_i(x)=0$ and $\lim_{x\to\8}x\te_i'(x)=0$ for every $i=1, 2, 3$. The proof will be completed, if we elaborate the case $c=1$. We know that $x\vp''(x)=\vp'(x)\te_2(x)$,
with
\begin{align*}
  \te_2(x)&=-\frac{\vartheta(\vp(x))}{1+\vartheta(\vp(x))}-
  \frac{\vartheta'(\vp(x))\vp(x)}{(1+\vartheta(\vp(x)))^2}
  =\vartheta(\vp(x))\left(-\frac{1}{1+\vartheta(\vp(x))}-
  \frac{\vartheta'(\vp(x))\vp(x)}{\vartheta(\vp(x))(1+\vartheta(\vp(x)))^2}\right).
\end{align*}
Therefore \eqref{fiequat1} is proved with $\s(x)=\vt(\vp(x))$ and
$$\t(x)=-\left(\frac{1}{1+\vartheta(\vp(x))}+
\frac{\vartheta'(\vp(x))\vp(x)}{\vartheta(\vp(x))(1+\vartheta(\vp(x)))^2}\right).$$
In order to show that $\s(2x)\simeq\s(x)$ it is enough to prove that $\vt(2x)\simeq\vt(x)$.
Notice that for some $\xi_x\in(0, 1)$ we have
\begin{align*}
  \left|\frac{\vt(2x)}{\vt(x)}-1\right|=\left|\frac{(x+\xi_xx)\vt'(x+\xi_xx)}{\vt(x+\xi_xx)}\right|
  \frac{x}{x+\xi_xx}\frac{\vt(x+\xi_xx)}{\vt(x)}\le\left|\frac{(x+\xi_xx)\vt'(x+\xi_xx)}{\vt(x+\xi_xx)}\right|
  \ _{\overrightarrow{x\to\8}}\ 0,
\end{align*}
since $\vt(x)$ is decreasing.
It is easy to see that $$\s(x)^{-1}\lesssim x^{\e}, \ \ \mbox{for every $\e>0$,}$$
since $\vt(x)^{-1}\lesssim_{\varepsilon}x^{\varepsilon}$ for every $\varepsilon>0$ and by \eqref{ratefi}. Furthermore, there exist $0<c_3\le c_4$ such that $c_3\le -\t(x)\le c_4$ for every $x\ge h(x_0)$, by \eqref{eq4}. The only what is left is to verify that $\lim_{x\to\8}\frac{x\te_2'(x)}{\te_2(x)}=0$ and $\lim_{x\to\8}x\te_3'(x)=0$. Indeed,
\begin{align*}
  \lim_{x\to\8}\frac{x\te_2'(x)}{\te_2(x)}=\lim_{x\to\8}\frac{\frac{(\vt''(\vp(x))\vp(x)^2+2\vt'(\vp(x))\vp(x))
  (1+\vartheta(\vp(x)))
  -2\vt'(\vp(x))^2\vp(x)^2}{\vartheta(\vp(x))(1+\vartheta(\vp(x)))^4}}{\frac{1}{1+\vartheta(\vp(x))}+
  \frac{\vartheta'(\vp(x))\vp(x)}{\vartheta(\vp(x))(1+\vartheta(\vp(x)))^2}}=0.
\end{align*}
In order to show that $\lim_{x\to\8}x\te_3'(x)=0$ it suffices to prove that
\begin{align*}
  \lim_{x\to\8}x\te_3'(x)=\lim_{x\to\8}\frac{\Theta_3(\vp(x))\vp(x)}{1+\vt(\vp(x))}=0,
\end{align*}
but this follows from \eqref{eq4} and \eqref{Tet3}, since $\lim_{x\to\8}x\Theta_3'(x)=0$. This completes the proof.
\end{proof}

\section{Estimates for some exponential sums}\label{section3}
The aim of this section is to establish Lemma \ref{vdc:lem1} and \ref{vdc:lem2} which will be essential for us and will be applied repeatedly in the sequel. Both proofs are based on Van der Corput's type estimates. In this section we will assume that $c\in[1, 4/3)$, $\g=1/c$, $h\in\mathcal{F}_c$ and $\vp$ is the inverse function to $h$.
 \begin{lem}[Van der Corput]\label{vdc}
Assume that $a, b\in\R$ and $a<b$. Let $F\in\mathcal{C}^2([a, b])$ be a real valued function and let $I$ be a subinterval of $[a, b]$. If there exists $\lambda>0$ and $r\ge 1$  such that
\begin{align*}
  \lambda\lesssim |F''(x)|\lesssim r\lambda,\ \ \mbox{for every \ $x\in I$,}
\end{align*}
then
$$\Big|\sum_{k\in I}e^{2\pi i F(k)}\Big|\lesssim r|I|\lambda^{1/2}+\lambda^{-1/2}.$$
\end{lem}
Proof of Lemma \ref{vdc} can be found in \cite{IK}, see Corollary 8.13, page 208. Lemma \ref{vdc:lem1} is a rather straightforward application of Lemma \ref{vdc}, whereas the estimate given in Lemma \ref{vdc:lem2} is more involved and its proof will explore brilliant ideas from \cite{UZ}.

Throughout the paper, we will use the following version  of summation by parts.
\begin{lem}\label{sbp}
Let $u(n)$ and $g(n)$ be arithmetic functions and $a, b\in\Z$ such that $0\le a<b$. Define the sum function
$U_a(t)=\sum_{a+1\le n\le t}u(n), \ \ \mbox{for any \ $t\ge a+1$}.$ Then
 \begin{align}\label{sbp1}
  \sum_{n=a+1}^bu(n)g(n)=U_a(b)g(b)-\sum_{n=a+1}^{b-1}U_a(n)(g(n+1)-g(n)).
\end{align}
Let $x$ and $y$ be real numbers such that $0\le y<x$. If $g\in\mathcal{C}^1([y, x])$, then
\begin{align}\label{sbp2}
 \sum_{y<n\le x}u(n)g(n)=U_{\lfloor y\rfloor}(x)g(x)-\int_y^xU_{\lfloor y\rfloor}(t)g'(t)dt.
\end{align}
\end{lem}
We encourage the reader to compare Lemma \ref{sbp} with \cite{Nat} Theorem A.4, page 304. In the sequel we will use the following identity.
\begin{align}\label{functfi}
  t^2\vp''(t)=\left\{ \begin{array} {ll}
\vp(t)(\g+\te_1(t))(\g-1+\te_2(t)), & \mbox{if $c>1$,}\\
\vp(t)(\g+\te_1(t))\s(t)\t(t), & \mbox{if $c=1$.}
\end{array}
\right.
\end{align}
\begin{lem}\label{vdc:lem1}
Assume that $N\ge1$, $x\in\Z$, $\a\in[0, 1]$, $m\in\Z\setminus\{0\}$, $l\ge1$ and $p, q\in\{0, 1\}$. If $N_{1, x}=\max\{N/2, N/2-x\}$, $N_{2, x}=\min\{4N, 4N-x\}$ then
\begin{align}\label{vdc:lem1e1}
 \bigg|\sum_{N_{1, x}<n\le N'\le N_{2, x}}e^{2\pi i(\a ln+m\vp(n+px+q))}\bigg|\lesssim |m|^{1/2}N\big(\vp(N)\s(N)\big)^{-1/2}.
\end{align}
For $c>1$ (see Section \ref{section2}) $\s$ is constantly equal to $1$.
\end{lem}
\begin{proof}
We shall apply Lemma \ref{vdc} to the exponential sum in \eqref{vdc:lem1e1}. We can assume, without loss of generality, that $m>0$ and let  $F(t)=\a lt+m\vp(t+px+q)$
for $t\in (N_{1, x}, N_{2, x}]$. According to \eqref{functfi} we see that
\begin{align*}
  |F''(t)|=|m\vp''(t+px+q)|\simeq\frac{m\vp(t+px+q)\s(t+px+q)}{(t+px+q)^2}
  \simeq\frac{m\vp(N)\s(N)}{N^2},
\end{align*}
since $p, q\in\{0, 1\}$, $N/2<t+px+q\le5N$, $\vp(2x)\simeq\vp(x)$ and $\s(2x)\simeq\s(x)$. One can think that $\s$ is constantly equal to $1$, when $c>1$ (see Section \ref{section2}). Now by Lemma \ref{vdc}
we obtain
\begin{align*}
 \bigg|\sum_{N_{1, x}<n\le N'\le N_{2, x}}e^{2\pi i(\a ln+m\vp(n+px+q))}\bigg|&\lesssim
 N\cdot\frac{m^{1/2}\vp(N)^{1/2}\s(N)^{1/2}}{N}+\frac{N}{m^{1/2}\vp(N)^{1/2}\s(N)^{1/2}}\\
 &\lesssim m^{1/2}N\big(\vp(N)\s(N)\big)^{-1/2},
\end{align*}
and the proof of \eqref{vdc:lem1e1} follows.
\end{proof}

\begin{lem}\label{vdc:lem2}
Assume that $N\ge1$, $x\in\Z$, $\a\in[0, 1]$, $m_1, m_2\in\Z\setminus\{0\}$, and $l\ge1$. Let $N_{1, x}=\max\{N/2, N/2-x\}$, $N_{2, x}=\min\{4N, 4N-x\}$ and $m=\max\{|m_1|, |m_2|\}$. If $x\ge\vp(N)^{\kappa}$ for some $\kappa\in[0, 1]$, then
\begin{align}\label{vdc:lem2e1}
 \bigg|\sum_{N_{1, x}<n\le N'\le N_{2, x}}e^{2\pi i(\a ln+m_1\vp(n)+m_2\vp(n+x))}\bigg|\lesssim m^{2/3}N^{4/3}\s(N)^{-1/3}\vp(N)^{-(1+\kappa)/3}.
\end{align}
\end{lem}

\begin{proof}
We shall apply Lemma \ref{vdc} to the exponential sum in \eqref{vdc:lem2e1}. Let $F(t)=\a lt+m_1\vp(t)+m_2\vp(t+x)$ for $t\in (N_{1, x}, N_{2, x}]$. Notice that according to \eqref{functfi} we have
\begin{align}\label{upest}
  |F''(t)|=|m_1\vp''(t)+m_2\vp''(t+x)|\lesssim \frac{m\vp(N)\s(N)}{N^2},
\end{align}
since $t, t+x\simeq N$, (if $c>1$ one can think that $\s$ is constantly equal to $1$). The lower bound for $|F''(t)|$ is much harder. We will follow the ideas from \cite{UZ} and we are going to prove that there exists $t_0\in(N_{1, x}, N_{2, x}]$ such that if
\begin{align*}
  |t-t_0|\ge N_0,\ \ \mbox{where}\ \ N_0=m^{a_1}N^{a_2}\s(N)^{-a_3}\vp(N)^{-a_4},
\end{align*}
for some $a_1, a_2, a_3, a_4\in\R$ which will be chosen later, then
\begin{align}\label{lowest}
  |F''(t)|\gtrsim m^{a_1}N^{a_2-4}\s(N)^{1-a_3}\vp(N)^{1+\kappa-a_4}.
\end{align}

Assume for a moment that \eqref{lowest} has been proved and let us finish the proof of \eqref{vdc:lem2e1}.
Combining \eqref{lowest} with \eqref{upest} we see
\begin{multline*}
  m^{a_1}N^{a_2-4}\s(N)^{1-a_3}\vp(N)^{1+\kappa-a_4}\lesssim|F''(t)|\\
  \lesssim
   m^{1-a_1}N^{2-a_2}\s(N)^{a_3}\vp(N)^{a_4-\kappa}\cdot m^{a_1}N^{a_2-4}\s(N)^{1-a_3}\vp(N)^{1+\kappa-a_4},
\end{multline*}
and Lemma \ref{vdc} can be applied with $r=m^{1-a_1}N^{2-a_2}\s(N)^{a_3}\vp(N)^{a_4-\kappa}$ and $\lambda=m^{a_1}N^{a_2-4}\s(N)^{1-a_3}\vp(N)^{1+\kappa-a_4}$. Indeed,
denote by $U(N')$ the sum in \eqref{vdc:lem2e1} and observe that
\begin{align*}
   |U(N')|
   \le \sum_{j=1}^3\bigg|\sum_{n\in A_j}e^{2\pi i(\a ln+m_1\vp(n)+m_2\vp(n+x))}\bigg|,
 \end{align*}
 where $A_1=(N_{1, x}, \min\{N', t_0-N_0\}]$, $A_2=(\max\{t_0+N_0, N_{1, x}\}, N']$ and $A_3=(\min\{N', t_0-N_0\}, \max\{t_0+N_0, N_{1, x}\}]$.
We shall apply Lemma \ref{vdc} to the first two sums, whereas the third one can be trivially estimated by $N_0$
if necessary, i.e. if $A_3\not=\emptyset$. Namely, we get
\begin{multline*}
  |U(N')|
  \lesssim N_0
  +N\cdot m^{1-a_1}N^{2-a_2}\s(N)^{a_3}\vp(N)^{a_4-\kappa}\cdot
  \big(m^{a_1}N^{a_2-4}\s(N)^{1-a_3}\vp(N)^{1+\kappa-a_4}\big)^{1/2}\\
  +\big(m^{a_1}N^{a_2-4}\s(N)^{1-a_3}\vp(N)^{1+\kappa-a_4}\big)^{-1/2}\\
  \lesssim N_0+m^{1-a_1/2}N^{1-a_2/2}\s(N)^{1/2+a_3/2}\vp(N)^{a_4/2+1/2-\kappa/2}\\
  +m^{-a_1/2}N^{2-a_2/2}\s(N)^{-1/2+a_3/2}\vp(N)^{a_4/2-1/2-\kappa/2}\\
  \lesssim N_0+m^{1-a_1/2}N^{2-a_2/2}\s(N)^{-1/2+a_3/2}\vp(N)^{a_4/2-1/2-\kappa/2}\\
  \lesssim m^{a_1}N^{a_2}\s(N)^{-a_3}\vp(N)^{-a_4}+m^{1-a_1/2}N^{2-a_2/2}\s(N)^{-1/2+a_3/2}\vp(N)^{a_4/2-1/2-\kappa/2}\\
  \lesssim m^{2/3}N^{4/3}\s(N)^{-1/3}\vp(N)^{-(1+\kappa)/3},
\end{multline*}
since the penultimate line forces some restrictions on $a_1, a_2, a_3, a_4$. Namely,
\begin{align*}
  a_1=1-a_1/2&\Longleftrightarrow a_1=2/3,\\
  a_2=2-a_2/2&\Longleftrightarrow a_2=4/3,\\
  -a_3=-1/2+a_3/2&\Longleftrightarrow a_3=1/3,\\
  -a_4=a_4/2-1/2-\kappa/2&\Longleftrightarrow a_4=(1+\kappa)/3.
\end{align*}
This proves \eqref{vdc:lem2e1}. Now what is left is to prove \eqref{lowest}. For this purpose we will proceed  as follows. Let
\begin{align*}
  F''(t)=\vp''(t)A(t),\ \ \mbox{where}\ \ A(t)=\left(m_1+m_2\frac{\vp''(t+x)}{\vp''(t)}\right).
\end{align*}
Let $\e'>0$ be a small enough real number whose precise value will be specified later.
If $|A(t)|\ge \e'm^{a_1}N^{a_2-2}\s(N)^{-a_3}\vp(N)^{\kappa-a_4}$ for every $t\in (N_{1, x}, N_{2, x}]$, then
\begin{align*}
  |F''(t)|=|\vp''(t)||A(t)|\gtrsim \e'm^{a_1}N^{a_2-4}\s(N)^{1-a_3}\vp(N)^{1+\kappa-a_4}.
\end{align*}
Assume now that there is some $t_0\in (N_{1, x}, N_{2, x}]$ such that $|A(t_0)|\le \e'm^{a_1}N^{a_2-2}\s(N)^{-a_3}\vp(N)^{\kappa-a_4}$. By the mean value theorem there is $\xi\in(0, 1)$ such that
\begin{align*}
 |A(t)-A(t_0)|=|t-t_0||A'(\xi_{t, t_0})|,
\end{align*}
where $\xi_{t, t_0}=t+\xi(t_0-t)$, if $t_0\ge t$ and $\xi_{t, t_0}=t_0+\xi(t-t_0)$, if $t_0<t$. In both cases
$\xi_{t, t_0}\simeq N$ and
$\xi_{t, t_0}+x\simeq N$, since $1\le\vp(N)^{\kappa}\le x\le 4N$.
Thus it is enough to estimate $|A'(t)|$ from below for any $t\simeq N$. Indeed, again by the mean value theorem, we see that for some $\xi_x\in(0, 1)$ we have
\begin{align*}
  A'(t)&=m_2\frac{\vp'''(t+x)\vp''(t)-\vp''(t+x)\vp'''(t)}{\vp''(t)^2}\\
  &=m_2\frac{\vp''(t+x)}{\vp''(t)}
  \left(\frac{\g-2+\te_3(t+x)}{t+x}-\frac{\g-2+\te_3(t)}{t}\right)\\
  &=xm_2\frac{\vp''(t+x)}{\vp''(t)}
  \left(\frac{2-\g-\te_3(t+\xi_xx)+
  (t+\xi_xx)\te_3'(t+\xi_xx)}{(t+\xi_xx)^2}\right).
\end{align*}
Therefore, there is a universal constant $C>0$ such that for sufficiently large $N\in\N$, by Lemma \ref{funlemfi}, we have
\begin{multline*}
  |A'(t)|
  \ge x\left|\frac{\vp''(t+x)}{\vp''(t)}\right|
  \left(\frac{2-\g}{(t+\xi_xx)^2}-\frac{|\te_3(t+\xi_xx)|+|(t+\xi_xx)\te_3'(t+\xi_xx)|}{(t+\xi_xx)^2}\right)\\
  \ge x\left|\frac{\vp''(t+x)}{\vp''(t)}\right|
  \frac{2-\g}{2(t+\xi_xx)^2}
  \ge C\frac{\vp(N)^{\kappa}}{N^2},
\end{multline*}
since $(t+\xi_xx)\simeq N$ and $|\te_3(t+\xi_xx)|+|(t+\xi_xx)\te_3'(t+\xi_xx)|\le (2-\g)/2$ for sufficiently large $N\in\N$.
This implies that, if $|t-t_0|\ge N_0=m^{a_1}N^{a_2}\s(N)^{-a_3}\vp(N)^{-a_4}$, then
\begin{align*}
 |A(t)-A(t_0)|=|t-t_0||A'(\xi_{t, t_0})|\ge C m^{a_1}N^{a_2-2}\s(N)^{-a_3}\vp(N)^{\kappa-a_4}.
\end{align*}
Finally, taking $\e'=C/2$, we obtain that
\begin{align*}
 |A(t)|\gtrsim m^{a_1}N^{a_2-2}\s(N)^{-a_3}\vp(N)^{\kappa-a_4},
\end{align*}
for every $|t-t_0|\ge N_0=m^{a_1}N^{a_2}\s(N)^{-a_3}\vp(N)^{-a_4}$ as desired and the proof of Lemma \ref{vdc:lem2} is completed.
\end{proof}

Now we have some refinements of Lemma \ref{vdc:lem1} and Lemma \ref{vdc:lem2}.
\begin{cor}\label{vdccor}
Assume that $N\ge1$, $x\in\Z$, $\a\in[0, 1]$, $m_1, m_2\in\Z\setminus\{0\}$, $l\ge1$ and $p, q\in\{0, 1\}$. Let $N_{1, x}=\max\{N/2, N/2-x\}$, $N_{2, x}=\min\{4N, 4N-x\}$. Then
\begin{multline}\label{vdccor1}
 \bigg|\sum_{N_{1, x}<n\le N'\le N_{2, x}}e^{2\pi i(\a ln+m_1\vp(n+px+q))}F_{m_1}^x(n)\bigg|\\
 \lesssim |m_1|^{1/2}N\big(\vp(N)\s(N)\big)^{-1/2}\Big(\sup_{n\in(N_{1, x}, N_{2, x}]}|F_{m_1}^x(n)|+N\sup_{n\in(N_{1, x}, N_{2, x}]}\left|F_{m_1}^x(n+1)-F_{m_1}^x(n)\right|\Big),
\end{multline}
where $F_{m_1}^x(n)$ is an arithmetic function.
If $x\ge\vp(N)^{\kappa}$ for some $\kappa\in[0, 1]$ and $m=\max\{|m_1|, |m_2|\}$, then
\begin{multline}\label{vdccor2}
 \bigg|\sum_{N_{1, x}<n\le N'\le N_{2, x}}e^{2\pi i(\a ln+m_1\vp(n)+m_2\vp(n+x))}F_{m_1, m_2}^x(n)\bigg|
 \lesssim m^{2/3}N^{4/3}\s(N)^{-\frac{1}{3}}\vp(N)^{-\frac{1+\kappa}{3}}\\
 \cdot\Big(\sup_{n\in(N_{1, x}, N_{2, x}]}|F_{m_1, m_2}^x(n)|
 +N\sup_{n\in(N_{1, x}, N_{2, x}]}\left|F_{m_1, m_2}^x(n+1)-F_{m_1, m_2}^x(n)\right|\Big),
\end{multline}
where $F_{m_1, m_2}^x(n)$ is an arithmetic function. For $c>1$ (see Section \ref{section2}) one may think that $\s$ is constantly equal to $1$.
\end{cor}
\begin{proof}
Let $U(N')$ denote the sum in \eqref{vdccor1} or \eqref{vdccor2} and $U_{N_{1, x}}(N')$ denotes the sum in \eqref{vdc:lem1e1} or \eqref{vdc:lem2e1} respectively. Finally, let $F(n)=F_{m_1}^x(n)$ or $F(n)=F_{m_1, m_2}^x(n)$. It is enough to apply \eqref{sbp1} to  $U(N')$. Namely, we have
\begin{align*}
  |U(N')|
  &\le\sup_{N'\in(N_{1, x}, N_{2, x}]}|U_{N_{1, x}}(N')F(N')|+\sum_{n=\lfloor N_{1, x}\rfloor+1}^{\lfloor N_{2, x}\rfloor}|U_{N_{1, x}}(n)|\left|F(n+1)-F(n)\right|\\
  &\lesssim\sup_{N'\in(N_{1, x}, N_{2, x}]}|U_{N_{1, x}}(N')|\Big(\sup_{n\in(N_{1, x}, N_{2, x}]}|F(n)|+N\sup_{n\in(N_{1, x}, N_{2, x}]}\left|F(n+1)-F(n)\right|\Big),
\end{align*}
and the proof follows from Lemma \ref{vdc:lem1} and \ref{vdc:lem2} respectively.
\end{proof}
We will show some application of Corollary \ref{vdccor}.  For this purpose let us define $\Phi(x)=\{x\}-1/2$ and expand $\Phi$ in the Fourier series (see \cite{HB} Section 2), i.e. we obtain
\begin{align}\label{four1}
  \Phi(t)=\sum_{0<|m|\le M}\frac{1}{2\pi i m}e^{-2\pi imt}+O\left(\min\left\{1, \frac{1}{M\|t\|}\right\}\right),
\end{align}
for $M>0$, where $\|t\|=\min_{n\in\Z}|t-n|$ is the distance of $t\in\R$ to the nearest integer. Parameter $M$ will give us some margin of flexibility in our further calculations and will allow us to produce the estimates with  the decay acceptable for us. Moreover,
\begin{align}\label{four2}
  \min\left\{1, \frac{1}{M\|t\|}\right\}=\sum_{m\in\Z}b_m e^{2\pi imt},
\end{align}
where
\begin{align}\label{fcoe2}
  |b_m|\lesssim \min\left\{\frac{\log M}{M}, \frac{1}{|m|}, \frac{M}{|m|^2}\right\}.
\end{align}

\begin{lem}\label{lem:3}
Assume that $N\ge1$, $p, q\in\{0, 1\}$, $x\in\Z$ and take $M\ge1$. Then
\begin{align}\label{lem:3e1}
  \sum_{n\in\N}\min\left\{1, \frac{1}{M\|\vp(n+px+q)\|}\right\}\eta\left(\frac {n}{N}\right)\eta\left(\frac{n+x}{N}\right)\lesssim \frac{N\log M}{M}+\frac{NM^{1/2}\log M}{(\s(N)\vp(N))^{1/2}},
\end{align}
where $\eta$ is a cut--off function as in Theorem \ref{thm:1max}.
\end{lem}
\begin{proof}
Let $S$ denote the sum in \eqref{lem:3e1}. Now we see, according to \eqref{four2}, that
\begin{align*}
S&=\sum_{n\in\N}\sum_{m\in\Z}b_m e^{2\pi im\vp(n+px+q)}\eta\left(\frac {n}{N}\right)\eta\left(\frac{n+x}{N}\right)\\
 &\lesssim\sum_{m\in\Z}|b_m|\bigg|\sum_{n\in\N} e^{2\pi im\vp(n+px+q)}\eta\left(\frac {n}{N}\right)\eta\left(\frac{n+x}{N}\right)\bigg|.
\end{align*}
Using \eqref{vdccor1} with $F_m^x(n)=\eta\left(\frac {n}{N}\right)\eta\left(\frac{n+x}{N}\right)$ and bounds
\eqref{fcoe2} for $|b_m|$ we immediately obtain
\begin{align*}
  \sum_{m\ge0}&|b_m|\bigg|\sum_{n\in\N} e^{2\pi im\vp(n+px+q)}\eta\left(\frac {n}{N}\right)\eta\left(\frac{n+x}{N}\right)\bigg|\\
  &\lesssim \frac{N\ \log M}{M} +
  \bigg(\sum_{0<m\le M}+\sum_{m> M}\bigg)|b_m|\frac{m^{1/2}N}{\big(\s(N)\vp(N)\big)^{1/2}}\\
  &\lesssim\frac{N\ \log M}{M}+\sum_{0<m\le M}m^{1/2}\frac{\log M}{M}\frac{N}{\big(\s(N)\vp(N)\big)^{1/2}}
  +\sum_{m> M}\frac{M}{m^{3/2}}\frac{N}{\big(\s(N)\vp(N)\big)^{1/2}}\\
  &\lesssim\frac{N\ \log M}{M}+\log M M^{1/2}\frac{N}{\big(\s(N)\vp(N)\big)^{1/2}},
\end{align*}
as desired.
\end{proof}
Now we have another application of Corollary \ref{vdccor} and Lemma \ref{lem:3}.
\begin{lem}\label{lemform1}
Assume that $h\in\mathcal{F}_c$, $\vp$ be its inverse and $\g=1/c$. If $0<\g\le 1$ and $\chi>0$ satisfy $4(1-\g)+6\chi<1$, then there exists $\e>0$ such that for every $N\in\N$ and for every $\a\in[0, 1]$
\begin{align}\label{form1}
   \sum_{n\in\mathbf{N}_{h}\cap[1, N]}\vp'(n)^{-1}e^{2\pi i \a n}=\sum_{n\in[1, N]}e^{2\pi i \a n}+O(N^{1-\chi-\e}).
\end{align}
The implied constant is independent of $\a$ and $N$.
\end{lem}
\begin{proof}
According to Lemma \ref{filem} (we may assume that it holds for all $n\in\mathbf{N}_h$) and the definition of function $\Phi(x)=\{x\}-1/2$ we obtain
\begin{align*}
  \sum_{n\in\mathbf{N}_{h}\cap[1, N]}&\vp'(n)^{-1}e^{2\pi i \a n}=\sum_{n\in[1, N]} \vp'(n)^{-1}\big(\lfloor-\vp(n)\rfloor-\lfloor-\vp(n+1)\rfloor\big)e^{2\pi i \a n}\\
  &=\sum_{n\in[1, N]} \vp'(n)^{-1}\big(\vp(n+1)-\vp(n)\big)e^{2\pi i \a n}\\
  &+\sum_{n\in[1, N]} \vp'(n)^{-1}\big(\Phi(-\vp(n+1))-\Phi(-\vp(n))\big)e^{2\pi i \a n}\\
  &=\sum_{n\in[1, N]} e^{2\pi i \a n}
  +\sum_{n\in[1, N]} \vp'(n)^{-1}\big(\Phi(-\vp(n+1))-\Phi(-\vp(n))\big)e^{2\pi i \a n}+O(\log N).
\end{align*}
The proof will completed if we show that
\begin{align}\label{form11}
  \sup_{P\in[1, N]}\bigg|\sum_{P<n\le P'\le 2P} \vp'(n)^{-1}\big(\Phi(-\vp(n+1))-\Phi(-\vp(n))\big)e^{2\pi i \a n}\bigg|\lesssim N^{1-\chi-\e}.
\end{align}
Let $S(P')$ denote the sum in \eqref{form11}. It is easy to see that the  Fourier expansions \eqref{four1} of $\Phi(x)$ leads us to
that
\begin{align*}
S(P')&=\sum_{0<|m|\le M}\frac{1}{2\pi i m}\sum_{P<n\le P'\le 2P} \vp'(n)^{-1}\Big(e^{2\pi i(\a n+m\vp(n+1))}-e^{2\pi i(\a n+m\vp(n))}\Big)\\
&+O\bigg(\sum_{P<n\le P'\le 2P}\vp'(n)^{-1}\left(\min\left\{1, \frac{1}{M\|\vp(n)\|}\right\}+\min\left\{1, \frac{1}{M\|\vp(n+1)\|}\right\}\right)\bigg)\\
&=\sum_{0<|m|\le M}\frac{1}{2\pi i m}\sum_{P<n\le P'\le 2P} e^{2\pi i(\a n+m\vp(n))}\vp'(n)^{-1}\Big(e^{2\pi im(\vp(n+1)-\vp(n))}-1\Big)\\
&+O\bigg(\frac{P^2\log M}{\vp(P)M}+\log M M^{1/2}\frac{P^2}{\s(P)^{1/2}\vp(P)^{3/2}}\bigg),
\end{align*}
with some $M\ge1$ which will be chosen later. Applying Corollary \ref{vdccor} to the inner sum in the penultimate expression and taking $M=P^{1+\chi+2\e}\vp(P)^{-1}$ (where $0<\e<\chi/10$ and $\chi>0$ such that $4(1-\g)+6\chi<1$) we get
\begin{align*}
S(P')&=O\bigg(M^{3/2}\frac{P}{\s(P)^{1/2}\vp(P)^{1/2}}+\frac{P^2\log M}{\vp(P)M}+\log M M^{1/2}\frac{P^2}{\s(P)^{1/2}\vp(P)^{3/2}}\bigg)\\
&=O\bigg(\frac{P^{5/2+3\chi/2+3\e}}{\s(P)^{1/2}\vp(P)^{2}}+P^{1-\chi-2\e}\log P+\log P \frac{P^{5/2+\chi/2+\e}}{\s(P)^{1/2}\vp(P)^{2}}\bigg)\\
&=O\big(P^{1-\chi-\e}\big(P^{3/2+5\chi/2+5\e-2\g}+P^{-\e}\log P\big)\big)=O\big(P^{1-\chi-\e}\big),
\end{align*}
since $3+5\chi+10\e-4\g<4(1-\g)+6\chi-1<0$.
\end{proof}

A straightforward application of formula \eqref{form1} with $\a=0$ shows that $|\mathbf{N}_h\cap[1, N]|\sim\vp(N)$. Indeed, let $U(x)=\sum_{n\in\mathbf{N}_{h}\cap[1, x]}\vp'(n)^{-1}$. Then, applying Lemma \ref{sbp} and \eqref{form1} we obtain
\begin{align*}
  |\mathbf{N}_h\cap[1, N]|&=\sum_{n\in\mathbf{N}_{h}\cap[1, N]}1=\sum_{n\in\mathbf{N}_{h}\cap[1, N]}\vp'(n)^{-1}\vp'(n)=U(N)\vp'(N)-\int_{1}^NU(x)\vp''(x)dx\\
   &=N\vp'(N)+O(\vp(N)N^{-\chi-\e})-\int_{1}^Nx\vp''(x)dx+O\bigg(\int_{1}^Nx^{1-\chi-\e}|\vp''(x)|dx\bigg)\\
   &=\vp(N)+O(\vp(N)N^{-\chi'}),
\end{align*}
for some $\chi'>0$, thus $|\mathbf{N}_h\cap[1, N]|\sim\vp(N)$.

\section{Proof of Theorem \ref{ergthm}}\label{secterg}
The main aim of this section is to prove Theorem \ref{ergthm}. For this purpose we will proceed as follows.
First of all we show the pointwise convergence on $L^2(X,\mu)$ using Lemma \ref{lemform1}, then by Theorem \ref{thm:1max}, interpolation and standard density argument, we extend this result for all $f\in L^p(X,\mu)$, where $p\ge1$.  We start from very simple observation based on summation by parts. Namely, if
\begin{align}\label{ergcon1}
  A_{h, N}^1f(x)=\frac{1}{N}\sum_{n\in\mathbf{N}_{h}\cap[1, N]} \vp'(n)^{-1}\ f(T^nx) \ _{\overrightarrow{N\to\8}}\ f^*(x) \ \ \mbox{for $\mu$ -- a.e. $x\in X$,}
\end{align}
then
\begin{align}\label{ergcon2}
 A_{h, N}f(x)=\frac{1}{|\mathbf{N}_{h}\cap[1, N]|}\sum_{n\in\mathbf{N}_{h}\cap[1, N]}f(T^nx) \ _{\overrightarrow{N\to\8}}\ f^*(x) \ \ \mbox{for $\mu$ -- a.e. $x\in X$.}
\end{align}
Let $M_kf(x)=\sum_{n\in\mathbf{N}_{h}\cap[1, k]}f(T^nx)$ and $M_k^1f(x)=\sum_{n\in\mathbf{N}_{h}\cap[1, k]} \vp'(n)^{-1}\ f(T^nx)$ and $M_0f(x)=M_0^1f(x)=0$. Let $m_k=\sum_{n\in\mathbf{N}_{h}\cap[1, k]}1=|\mathbf{N}_{h}\cap[1, k]|\sim\vp(k)$ and
$m_k^1=\sum_{n\in\mathbf{N}_{h}\cap[1, k]}\vp'(n)^{-1}$. Then, for $f\ge0$, we have
\begin{multline*}
  A_{h, N}f(x)=
  \frac{1}{|\mathbf{N}_{h}\cap[1, N]|}\sum_{n\in\mathbf{N}_{h}\cap[1, N]}\vp'(n)\vp'(n)^{-1}f(T^nx)\\
  =\frac{1}{|\mathbf{N}_{h}\cap[1, N]|}\sum_{k=1}^N\vp'(k)\big(M_k^1f(x)-M_{k-1}^1f(x)\big)\\
  =\frac{\vp'(N)}{|\mathbf{N}_{h}\cap[1, N]|}M_N^1f(x)+
  \frac{1}{|\mathbf{N}_{h}\cap[1, N]|}\sum_{k=1}^{N-1}\left(\vp'(k)-\vp'(k+1)\right)M_k^1f(x)\\
  =\frac{m_N^1\vp'(N)}{|\mathbf{N}_{h}\cap[1, N]|}\frac{N}{m_N^1}A_{h, N}^1f(x)
  +
  \frac{1}{|\mathbf{N}_{h}\cap[1, N]|}\sum_{k=1}^{N-1}\left(m_k^1\vp'(k)-m_k^1\vp'(k+1)\right)\frac{k}{m_k^1}A_{h, k}^1f(x).
\end{multline*}
On the other hand
\begin{align*}
  \frac{m_N^1\vp'(N)}{|\mathbf{N}_{h}\cap[1, N]|}f^*(x)+
  \frac{1}{|\mathbf{N}_{h}\cap[1, N]|}\sum_{k=1}^{N-1}\left(m_k^1\vp'(k)-m_k^1\vp'(k+1)\right)f^*(x)\\
  =\frac{1}{|\mathbf{N}_{h}\cap[1, N]|}\sum_{k=1}^{N}\vp'(k)\big(m_k^1-m_{k-1}^1\big)f^*(x)=f^*(x).
\end{align*}
Let $\e>0$ such that for every $N> N_0$ we have
$$\left|\frac{N}{m_N^1}A_{h, N}^1f(x)-f^*(x)\right|<\e.$$
Since, $|\mathbf{N}_{h}\cap[1, N]|\ _{\overrightarrow{N\to\8}}\ \8$, we see
\begin{multline*}
  \limsup_{N\to\8}\bigg|\frac{1}{|\mathbf{N}_{h}\cap[1, N]|}\sum_{n\in\mathbf{N}_{h}\cap[1, N]}f(T^nx)-f^*(x)\bigg|
  \le\limsup_{N\to\8}\frac{m_N^1\vp'(N)}{|\mathbf{N}_{h}\cap[1, N]|}\left|\frac{N}{m_N^1}A_{h, k}^1f(x)-f^*(x)\right|\\
  +
  \limsup_{N\to\8}\frac{1}{|\mathbf{N}_{h}\cap[1, N]|}\left(\sum_{k=1}^{N_0}+\sum_{k=N_0+1}^{N-1}\right)\left(m_k^1\vp'(k)-m_k^1\vp'(k+1)\right)\left|\frac{k}{m_k^1}A_{h, k}^1f(x)-f^*(x)\right|
  \le \e,
\end{multline*}
and \eqref{ergcon2} is justified.

In order to prove \eqref{ergcon1} on $L^2(X, \mu)$ it suffices to show that
\begin{align}\label{ergmax}
  \Big\|\sup_{N\in\mathcal{D}}|A_{h, N}^1f|\Big\|_{L^2(X, \mu)}\lesssim\|f\|_{L^2(X, \mu)},
\end{align}
where $\mathcal{D}=\{2^n: n\in\N\}$ and
\begin{align}\label{ergosc}
  \sum_{j=1}^J\Big\|\sup_{\genfrac{}{}{0pt}{}{N_j<N\le N_{j+1}}{N\in Z_{\eps}}}|A_{h, N}^1f-A_{h, N_j}^1f|\Big\|_{L^2(X, \mu)}\le o(J)\|f\|_{L^2(X, \mu)},
\end{align}
where $Z_{\e}=\{\lfloor(1+\eps)^n\rfloor: n\in\N\}$ for some fixed $\eps>0$ and $(N_j)_{j\in\N}$ is any rapidly
increasing sequence $2N_j<N_{j+1}$. Using transference principle as in \cite{B3} we see that \eqref{ergmax}
and \eqref{ergosc} can be transferred to $\Z$ and \eqref{ergmax} follows from Theorem \ref{thm:1max} by interpolation. If it comes to \eqref{ergosc} we use Lemma \ref{lemform1}. Indeed, let
\begin{align*}
  K^1_{h, N}(x)=\frac{1}{N}\sum_{n\in\mathbf{N}_{h}\cap[1, N]} \vp'(n)^{-1}\d_n(x), \ \ \mbox{and}\ \ \
  K^2_{h, N}(x)=\frac{1}{N}\sum_{n\in[1, N]}\d_n(x),
\end{align*}
for $x\in\Z$ and observe
\begin{align*}
  \sum_{j=1}^J\Big\|\sup_{\genfrac{}{}{0pt}{}{N_j<N\le N_{j+1}}{N\in Z_{\eps}}}|K^1_{h, N}*f-K^1_{h, N_j}*f|\Big\|_{\ell^2(\Z)}
  &\lesssim \sum_{j=1}^J\Big\|\sup_{\genfrac{}{}{0pt}{}{N_j<N\le N_{j+1}}{N\in Z_{\eps}}}|K^2_{h, N}*f-K^2_{h, N_j}*f|\Big\|_{\ell^2(\Z)}\\
  &+\sum_{j=1}^J\bigg(\sum_{\genfrac{}{}{0pt}{}{N_j<N\le N_{j+1}}{N\in Z_{\eps}}}\big\|K^1_{h, N}*f-K^2_{h, N}*f\big\|^2_{\ell^2(\Z)}\bigg)^{1/2}\\
  &\lesssim o(J)\|f\|_{\ell^2(\Z)}+\sum_{j=1}^J N_j^{-\chi}\|f\|_{\ell^2(\Z)}\lesssim o(J)\|f\|_{\ell^2(\Z)},
\end{align*}
as desired. Since the first inequality was proved in  \cite{B3}, and the second one follows from Parseval's identity and Lemma \ref{lemform1}.

\section{Necessary approximations}\label{section4}
This section is devoted to the study of properties of the kernel $K_{h, N}(x)$ as defined in \eqref{thm:1max2} or more precisely $K_{h, N}*\widetilde{K}_{h, N}(x)$, where $\widetilde{K}_{h, N}(x)=K_{h, N}(-x)$. We shall show that $K_{h, N}*\widetilde{K}_{h, N}(x)$ can be split into a delta mass at $0$, a slowly varying function $G_{N}(x)$, and a small error term $E_{N}(x)$. From now on we will assume that $29/30<\g=1/c<1$. The case when $c=1$ is unavailable at this moment due to the lack of decay of order $1/N$ in Lemma \ref{lem:1ker0e1}. The best  we could do is $1/N^{1-\e}$ for arbitrary $\e>0$. If Lemma \ref{lem:1ker0e1} was true in the case $c=1$,  it must have been based on  completely new ideas. Since $|\mathbf{N}_h\cap[1, N]|\sim\vp(N)$ we will replace $|\mathbf{N}_h\cap[1, N]|$ by $\vp(N)$ in the definition of $K_{h, N}(x)$. We begin with the following observation
\begin{multline*}
  K_{h, N}*\widetilde{K}_{h, N}(x)=\frac{1}{\vp(N)^2}\sum_{m\in\mathbf{N}_h}\sum_{n\in\mathbf{N}_h}
  \d_{m}*\d_{-n}(x)\eta\left(\frac m N\right)\eta\left(\frac n N\right)\\
  =\frac{1}{\vp(N)^2}\sum_{m\in\mathbf{N}_h}\sum_{n\in\mathbf{N}_h}\d_{m}(x+n)\eta\left(\frac m N\right)\eta\left(\frac n N\right)
  =\frac{1}{\vp(N)^2}\sum_{n\in\mathbf{N}_h}\mathbf{1}_{\mathbf{N}_h}(x+n)
  \eta\left(\frac{n+x}{N}\right)\eta\left(\frac n N\right).
\end{multline*}
This also proves that $K_{h, N}*\widetilde{K}_{h, N}(x)=K_{h, N}*\widetilde{K}_{h, N}(-x)$ for all $x\in\Z$.
\begin{lem}\label{lem:1ker0}
Assume that $0<|x|\le\vp(N)$, then
 \begin{align}\label{lem:1ker0e1}
  |K_{h, N}*\widetilde{K}_{h, N}(x)|\lesssim \frac{1}{N}.
\end{align}
\end{lem}
\begin{proof}
Here we will use the argument from \cite{UZ} to show \eqref{lem:1ker0e1}. We may assume that
$0<x\le\vp(N)$, since $K_{h, N}*\widetilde{K}_{h, N}(x)$ is symmetric. Observe that
$K_{h, N}*\widetilde{K}_{h, N}(x)$ is nonzero if and only if $n, n+x\in\mathbf{N}_h$ and $n, n+x\simeq N$.
Thus we have to count the number of such $n$'s uniformly with respect to  $1\le x\le\vp(N)$. Observe that
\begin{multline}\label{ineq1}
  K_{h, N}*\widetilde{K}_{h, N}(x)
  =\frac{1}{\vp(N)^2}\sum_{(s, k)\in\N\times\N}\mathbf{1}_{\lfloor h(k+s)\rfloor}(x+\lfloor h(k)\rfloor)
  \eta\left(\frac{\lfloor h(k+s)\rfloor}{N}\right)\eta\left(\frac {\lfloor h(k)\rfloor}{N}\right)\\
  \le\frac{1}{\vp(N)^2}\sum_{(s, k)\in A_{N}}
  \eta\left(\frac{\lfloor h(k+s)\rfloor}{N}\right)\eta\left(\frac {\lfloor h(k)\rfloor}{N}\right),
\end{multline}
where $A_N=\{(s, k)\in\N\times\N: \vp(N/2)\le k\le \vp(5N),\ s\simeq \frac{x\vp(N)}{N},\ x-1\le h(k+s)-h(k)\le x+1\}$. The last inequality can be achieved as follows. Recall that
$n\in\mathbf{N}_h$ if and only if $n=\lfloor h(k)\rfloor$ for some $k\in\N$, but if $N/2\le n=\lfloor h(k)\rfloor\le 4N$, then $\vp(N/2)\le k\le \vp(5N)$. Moreover, if $\lfloor h(k)\rfloor+x=\lfloor h(k+s)\rfloor$ holds for some $(s, k)\in\N\times\N$, then $x-1\le h(k+s)-h(k)\le x+1$ is satisfied for the same pairs. Finally,
define $g(k)=h(k+s)-h(k)$ and observe
\begin{align*}
  g(k)=h(k+s)-h(k)=\int_k^{k+s}h'(t)dt\simeq sh'(\vp(N))\ \Longrightarrow\ x\simeq sh'(\vp(N))\simeq\frac{sN}{\vp(N)},
\end{align*}
this implies that $s\simeq \frac{x\vp(N)}{N}$ and justifies \eqref{ineq1}.
The task now is to estimate the cardinality of $A_N$. For this purpose it suffices to find the distance
between $g(k+1)$ and $g(k)$, since $g(k)$ is increasing. We see that there are $\xi_1, \xi_2\in(0, 1)$ such that
\begin{align*}
  g(k+1)-g(k)&=h'(k+s+\xi_1)-h'(k+\xi_1)
  =sh''(k+\xi_1+\xi_2s)
  \simeq
 \frac{sN}{\vp(N)^2}\simeq\frac{x}{\vp(N)}\lesssim1,
\end{align*}
since also $0<s\lesssim\vp(N)$. Combining these observations we see that for a fixed $s$ such that $s\simeq \frac{x\vp(N)}{N}$
 we have at most $1+\left(\frac{sN}{\vp(N)^2}\right)^{-1}\simeq\frac{\vp(N)}{x}$ values of $k\simeq\vp(N)$
for which the inequality $x-1\le h(k+s)-h(k)\le x+1$ holds. Therefore,
\begin{align*}
  K_{h, N}*\widetilde{K}_{h, N}(x)\lesssim \frac{1}{\vp(N)^2}\sum_{(s, k)\in A_{N}}
  \eta\left(\frac{\lfloor h(k+s)\rfloor}{N}\right)\eta\left(\frac {\lfloor h(k)\rfloor}{N}\right)
  \lesssim \frac{1}{\vp(N)^2}\cdot\frac{x\vp(N)}{N}\cdot\frac{\vp(N)}{x}=\frac{1}{N}.
\end{align*}
This completes the proof of \eqref{lem:1ker0e1}.
\end{proof}

\begin{lem}\label{lem:1ker}
There exists $\chi>0$ such that $K_{h, N}*\widetilde{K}_{h, N}(x)=G_N(x)+E_N(x)$ for all $|x|>\vp(N)$, where $|E_N(x)|\lesssim N^{-1-\chi}$ and
\begin{align}\label{lem:1ker1}
  G_N(x)=\frac{1}{\vp(N)^2}\sum_{n\in\N}\vp'(n)\vp'(n+|x|)\eta\left(\frac n N\right)\eta\left(\frac{n+|x|}{N}\right).
\end{align}
Moreover, $|G_N(x)|\lesssim N^{-1}$ and $|G_N(x+h)-G_N(x)|\lesssim N^{-2}|h|$ for all $x, h\in\Z$.
\end{lem}
\begin{proof}
We may assume that $x>\vp(N)$, since $K_{h, N}*\widetilde{K}_{h, N}(x)$ is symmetric.
In order to prove \eqref{lem:1ker1} we apply Lemma \ref{intlem} and notice that for $l\in\N$
\begin{align}\label{lem:1ker2}
  \lfloor-\vp(l)\rfloor-\lfloor-\vp(l+1)\rfloor=\vp(l+1)-\vp(l)+\Phi(-\vp(l+1))-\Phi(-\vp(l)),
\end{align}
where $\Phi(x)=\{x\}-1/2$. Recalling \eqref{four1} let us introduce
\begin{align}\label{four3}
\begin{split}
  \Delta_M(n+x)&=\sum_{0<|m|\le M}\frac{1}{2\pi i m} \Big(e^{2\pi im\vp(n+x+1)}-e^{2\pi im\vp(n+x)}\Big),\\
  \Pi_M(n+x)&=\big(\Phi(-\vp(n+x+1))-\Phi(-\vp(n+x))\big)-\Delta_M(n+x). \ \ \mbox{Moreover,}
  \end{split}
  \end{align}
  \begin{align}
  \label{four31}|\Delta_{M}(n+x)|&\lesssim\log M,\\
  \label{four32}|\Pi_M(n+x)|&\lesssim\min\left\{1, \frac{1}{M\|\vp(n+x)\|}\right\}+\min\left\{1, \frac{1}{M\|\vp(n+x+1)\|}\right\}.
\end{align}
Observe that for every $l\in\N$ there is  $\xi_l\in(0, 1)$ such that $\vp(l+1)=\vp(l)+\vp'(l)+\vp''(l+\xi_l)/2.$
Combining all these things we have
\begin{align}\label{expans1}
  \mathbf{1}_{\mathbf{N}_h}&(n)\mathbf{1}_{\mathbf{N}_h}(n+x)=\big(\lfloor-\vp(n)\rfloor-\lfloor-\vp(n+1)\rfloor\big)
  \big(\lfloor-\vp(n+x)\rfloor-\lfloor-\vp(n+x+1)\rfloor\big)\\
 \nonumber &=\big(\vp(n+1)-\vp(n)\big)\big(\vp(n+x+1)-\vp(n+x)\big)\\
 \nonumber &+\big(\vp(n+1)-\vp(n)\big)\big(\Phi(-\vp(n+x+1))-\Phi(-\vp(n+x))\big)\\
 \nonumber &+\big(\vp(n+x+1)-\vp(n+x)\big)\big(\Phi(-\vp(n+1))-\Phi(-\vp(n))\big)\\
 \nonumber &+\big(\Phi(-\vp(n+1))-\Phi(-\vp(n))\big)\big(\Phi(-\vp(n+x+1))-\Phi(-\vp(n+x))\big)\\
 \nonumber &=\vp'(n)\vp'(n+x)+\vp'(n)\Delta_M(n+x)+\vp'(n+x)\Delta_M(n)+\big(\vp'(n)\Pi_M(n+x)\\
 \nonumber &+\vp'(n+x)\Pi_M(n)\big)+\Delta_M(n)\Delta_M(n+x)+\big(\Delta_M(n)\Pi_M(n+x)+\Pi_M(n)\Delta_M(n+x)\big)\\
 \nonumber &+\Pi_M(n)\Pi_M(n+x)+\Theta(n, x),
\end{align}
where
\begin{multline}\label{defth}
  \Theta(n, x)=\vp'(n)\vp''(n+x+\xi_{n+x})/2+\vp''(n+\xi_n)\vp'(n+x)/2\\
  +\vp''(n+\xi_n)\vp''(n+x+\xi_{n+x})/4+\vp''(n+\xi_n)\big(\Phi(-\vp(n+x+1))-\Phi(-\vp(n+x))\big)/2\\
 +\vp''(n+x+\xi_{n+x})\big(\Phi(-\vp(n+1))-\Phi(-\vp(n))\big)/2.
\end{multline}
Therefore, according to \eqref{expans1} we have
\begin{align*}
  K_{h, N}*\widetilde{K}_{h, N}(x)=\sum_{j=1}^{8}I_j(x),
\end{align*}
where
\begin{align*}
  I_1(x)=\frac{1}{\vp(N)^2}\sum_{n\in\N}\vp'(n)\vp'(n+x)\eta\left(\frac n N\right)\eta\left(\frac{n+x}{N}\right),
\end{align*}
\begin{align*}
  I_2(x)=\frac{1}{\vp(N)^2}\sum_{0<|m|\le M}\frac{1}{2\pi i m}\sum_{n\in\N}e^{2\pi im\vp(n+x)}\cdot\Psi_2(m, n, x),
\end{align*}
where $\Psi_2(m, n, x)=\vp'(n)\big(e^{2\pi im(\vp(n+x+1)-\vp(n+x))}-1\big)\eta\left(\frac n N\right)\eta\left(\frac{n+x}{N}\right)$.
\begin{align*}
  I_3(x)=\frac{1}{\vp(N)^2}\sum_{0<|m|\le M}\frac{1}{2\pi i m}\sum_{n\in\N}e^{2\pi im\vp(n)}\cdot\Psi_3(m, n, x),
\end{align*}
where $\Psi_3(m, n, x)=\vp'(n+x)\big(e^{2\pi im(\vp(n+1)-\vp(n))}-1\big)\eta\left(\frac n N\right)\eta\left(\frac{n+x}{N}\right)$.
\begin{align*}
  I_4(x)=\frac{1}{\vp(N)^2}\sum_{n\in\N}\big(\vp'(n)\Pi_M(n+x)+\vp'(n+x)\Pi_M(n)\big)\eta\left(\frac n N\right)\eta\left(\frac{n+x}{N}\right),
\end{align*}
\begin{align*}
  I_5(x)=\frac{1}{\vp(N)^2}\sum_{0<|m_1|\le M}&\sum_{0<|m_2|\le M}\frac{1}{(2\pi i)^2 m_1m_2}\sum_{n\in\N}e^{2\pi i(m_1\vp(n)+m_2\vp(n+x))}\cdot\Psi_5(m_1, m_2, n, x),
\end{align*}
where $\Psi_5(m_1, m_2, n, x)=\big(e^{2\pi im_1(\vp(n+1)-\vp(n))}-1\big)
   \big(e^{2\pi im_2(\vp(n+x+1)-\vp(n+x))}-1\big)\eta\left(\frac n N\right)\eta\left(\frac{n+x}{N}\right)$.
\begin{align*}
  I_6(x)=\frac{1}{\vp(N)^2}\sum_{n\in\N}\big(\Delta_M(n)\Pi_M(n+x)+\Pi_M(n)\Delta_M(n+x)\big)
  \eta\left(\frac n N\right)\eta\left(\frac{n+x}{N}\right),
\end{align*}
\begin{align*}
  I_7(x)=\frac{1}{\vp(N)^2}\sum_{n\in\N}\Pi_M(n)\Pi_M(n+x)\eta\left(\frac n N\right)\eta\left(\frac{n+x}{N}\right),
\end{align*}
\begin{align*}
  I_8(x)=\frac{1}{\vp(N)^2}\sum_{n\in\N}\Theta(n, x)\eta\left(\frac n N\right)\eta\left(\frac{n+x}{N}\right).
\end{align*}
Recall that $29/30<\g<1$, and let $M=N^{1+2\chi+\e}\vp(N)^{-1}$ for  $\chi=1-\g>0$ and some $0<\e<\chi/10$ and notice that
 \begin{align}\label{kappa}
   29/30<\g\ \Longleftrightarrow\ 10(1-\g)+20\chi<1.
 \end{align}
  The proof will be completed if we show that $|I_1(x)|\lesssim N^{-1}$ and $|I_1(x+h)-I_1(x)|\lesssim N^{-2}|h|$ for $x, h\in\Z$ and for every $2\le j\le8$ we have $|I_j(x)|\lesssim N^{-1-\chi}$ where $x>\vp(N)$. \\
  \noindent \textsf{\textbf{\underline{Estimates for $I_1(x)$.}}} Observe that
  \begin{align*}
  |I_1(x)|\lesssim\frac{1}{\vp(N)^2}\frac{\vp(N)^2N}{N^2}=\frac{1}{N},
\end{align*}
  and
  \begin{align*}
    |I_1(x+h)-I_1(x)|\lesssim\frac{1}{\vp(N)^2}\sum_{n\in\N}\vp'(n)\eta\left(\frac n N\right)\bigg|\int_{n+x}^{n+x+h}\vp''(t)\eta\left(\frac{t}{N}\right)
    +\vp'(t)\frac{1}{N}\eta'\left(\frac{t}{N}\right)dt\bigg|\\
    \lesssim \frac{|h|}{\vp(N)^2}\frac{\vp(N)^2N}{N^3}\lesssim \frac{|h|}{N^2},
  \end{align*}
  as claimed.\\
\noindent \textsf{\textbf{\underline{Estimates for $I_2(x), I_3(x), I_5(x)$.}}} Applying estimates \eqref{vdccor1} with $F_m^x(n)=\Psi_2(m, n, x)$ or $F_m^x(n)=\Psi_3(m, n, x)$ to the inner sum in $I_2(x)$
and $I_3(x)$ respectively we obtain that
\begin{align}\label{lem:1kerest1}
  |I_2(x)|+|I_3(x)|\lesssim  \frac{1}{\vp(N)^2}\sum_{0<m\le M}\frac{1}{m}\cdot\frac{m\vp(N)^2}{N^2}\cdot
  m^{1/2}N\big(\vp(N)\s(N)\big)^{-1/2}\\
 \nonumber \lesssim\frac{M^{3/2}}{N\vp(N)^{1/2}\s(N)^{1/2}}\lesssim \frac{1}{N^{1+\chi}}\frac{N^{3/2+4\chi+2\e}}{N^{2\g}}\lesssim \frac{1}{N^{1+\chi}},
\end{align}
since $|F_m^x(t)|\lesssim \frac{m\vp(N)^2}{N^2}$, $|\frac{d}{dt}F_m^x(t)|\lesssim \frac{m\vp(N)^2}{N^3}$, $\s(x)^{-1}\lesssim x^{\e_1}$ and $x^{\g-\e_1}\lesssim \vp(x)$ for any $\e_1>0$.
The last inequality in \eqref{lem:1kerest1} holds since by \eqref{kappa} we have
\begin{align*}
  3+8\chi+4\e-4\g=4(1-\g)+8\chi+4\e-1<10(1-\g)+20(1-\g)-1<0.
\end{align*}
Arguing in a similar way as above and applying \eqref{vdccor2} with $\kappa=1$ and $F_{m_1, m_2}^x(n)=\Psi_5(m_1, m_2, n, x)$
to the inner sum in $I_5(x)$ we obtain that
\begin{align}\label{lem:1kerest2}
  |I_5(x)|\lesssim  \frac{1}{\vp(N)^2}\sum_{0<m_1\le M}\sum_{0<m_2\le M}\frac{1}{m_1m_2}\cdot\frac{m_1m_2\vp(N)^2}{N^2}\cdot
  \frac{\max\{m_1, m_2\}^{2/3}N^{4/3}}{\vp(N)^{2/3}}\\
 \nonumber \lesssim\frac{M^{8/3}}{N^{2/3}\vp(N)^{2/3}}\lesssim
 \frac{1}{N^{1+\chi}}\frac{N^{3+19\chi/3+3\e}}{N^{10\g/3}}\lesssim \frac{1}{N^{1+\chi}},
\end{align}
since $|F_{m_1, m_2}^x(x)|\lesssim \frac{m_1m_2\vp(N)^2}{N^2}$, $|\frac{d}{dt}F_{m_1, m_2}^x(x)|\lesssim \frac{m_1m_2\vp(N)^2}{N^3}$. The last inequality in \eqref{lem:1kerest2} holds since by \eqref{kappa} we have
\begin{align*}
  9+19\chi+9\e-10\g<10(1-\g)+20\chi-1<0.
\end{align*}
\noindent \textsf{\textbf{\underline{Estimates for $I_4(x), I_6(x), I_7(x)$.}}} According to \eqref{four31}, \eqref{four32} and Lemma \ref{lem:3} we have
\begin{multline}\label{lem:1kerest3}
 |I_4(x)|+|I_6(x)|+|I_7(x)|\\
 \lesssim\frac{\log M}{\vp(N)^2}\sum_{p, q\in\{0,1\}}\sum_{n\in\N}\min\left\{1, \frac{1}{M\|\vp(n+px+q)\|}\right\}\eta\left(\frac {n}{N}\right)\eta\left(\frac{n+x}{N}\right)\\
 \lesssim \frac{N\log^2M}{M\vp(N)^2}+\frac{NM^{1/2}\log^2M}{\s(N)^{1/2}\vp(N)^{5/2}}\lesssim
 \frac{\log^2N}{N^{2\chi+\e}\vp(N)}+\frac{N^{3/2+\chi+\e/2}\log^2N}{\s(N)^{1/2}\vp(N)^{3}}\\
 \lesssim \frac{1}{N^{\chi+1-\g+\g}}+\frac{1}{N^{1+\chi}}\frac{N^{5/2+2\chi+2\e}}{N^{3\g}}\lesssim \frac{1}{N^{1+\chi}},
\end{multline}
since by \eqref{kappa} we have
\begin{align*}
  5+4\chi+4\e-6\g=6(1-\g)+4\chi+4\e-1<10(1-\g)+20\chi-1<0.
\end{align*}
\noindent \textsf{\textbf{\underline{Estimates for $I_8(x)$.}}} In view of definition \eqref{defth}  we get
\begin{align*}
  |I_8(x)|\lesssim\frac{1}{\vp(N)^2}\sum_{n\in\N}\frac{\vp(N)}{N^2}\eta\left(\frac n N\right)\eta\left(\frac{n+x}{N}\right)\lesssim\frac{1}{N\vp(N)}\lesssim\frac{1}{N^{1+\g-\e}}\lesssim\frac{1}{N^{1+\chi}},
  \end{align*}
  since $29/30<\g< 1$ and $0<\e<\chi/10=(1-\g)/10<2\g-1$ which in turn gives $\g-\e>\g+1-2\g=\chi$. The proof of Lemma \ref{lem:1ker} is completed.
\end{proof}

\section{Proof of Theorem \ref{thm:1max}}\label{sectmax}
The maximal functions which will occur in this section will be initially defined for any nonnegative finitely supported function $f\ge0$ and unless otherwise stated $f$ is always such a function.  Recall that
\begin{align*}
  \mathcal{M}_hf(x)=\sup_{N\in\mathcal{D}}|K_{h, N}*f(x)|,
\end{align*}
where $\mathcal{D}=\{2^n: n\in\N\}$ for $K_{h, N}$ defined in \eqref{thm:1max2} with normalizing factor $\vp(N)$ instead of $|\mathbf{N}_h\cap[1, N]|$, but this is not important here, since $|\mathbf{N}_h\cap[1, N]|\sim\vp(N)$.

Theorem \ref{thm:1max} will follow from Theorem \ref{thm:3CZ} which is stated in a more abstract way. The idea of proof of Theorem \ref{thm:3CZ} was pioneered by Fefferman \cite{F} and after that was applied to maximal functions in continuous settings in \cite{C}. Recently, it turned out that the method is flexible enough and was applied to study discrete maximal functions, see \cite{UZ}, \cite{LaV} and \cite{C1}.

The crucial role in the proof of Theorem \ref{thm:3CZ} will be played by Lemma \ref{lem:4} stated at the end of this section. Its proof will strongly exploit the nature of the kernel $K_n*\widetilde{K}_n(x)$, i.e.
\eqref{thm:3CZ2}, \eqref{thm:3CZ3} and \eqref{thm:3CZ4}. In our case these proporties will follow from Lemma \ref{lem:1ker0} and Lemma \ref{lem:1ker}.
\begin{thm}\label{thm:3CZ}
Let $\mathcal{M}f(x)=\sup_{n\in\N}|K_n*f(x)|$ be a maximal function corresponding with a family of nonnegative kernels $(K_n)_{n\in\N}\subseteq\ell^1(\Z)$ such that $\|\mathcal{M}f\|_{\ell^{\8}(\Z)}\lesssim\|f\|_{\ell^{\8}(\Z)}$ for all $f\in\ell^{\8}(\Z)$ and let $(F_n)_{n\in\N}$ be a family of nonnegative functions. Assume that there are sequences $(d_n)_{n\in\N}, (D_n)_{n\in\N}\subseteq[1, \8)$ such that $|\supp\ K_n|=d_n$, $\supp\ K_n\subseteq[0, D_n]$, $\supp\ F_n\subseteq[-D_n, D_n]$, $d_n\le  D_n^{\e_0}$ for some $\e_0\in(0, 1)$ and there is a finite constant $M>1$ such that $Md_n\le d_{n+1}, MD_n\le D_{n+1}$ for all $n\in\N$. Moreover, there exists $\e_1>0$ such that for every $n\in\N$ and $x\in\Z$ we have
\begin{align}\label{thm:3CZ2}
  |K_n*\widetilde{K}_n(x)-F_n(x)|\lesssim D_n^{-1-\e_1},
\end{align}
where $\widetilde{K}_n(x)=K_n(-x)$, and
\begin{align}\label{thm:3CZ3}
  F_n(0)\lesssim d_n^{-1}, \ \ \mbox{and}\ \ \ |F_n(x)|\lesssim D_n^{-1} \ \ \mbox{for every $x\not=0$}.
\end{align}
Finally, for some $\e_2\in(0, 1]$ we have
\begin{align}\label{thm:3CZ4}
  |F_n(x+y)-F_n(x)|\lesssim D_n^{-2}|y|, \ \ \mbox{whenever \ $|x|, |x+y|\gtrsim d_n^{\e_2}.$}
\end{align}
Then
\begin{align}\label{thm:3CZ5}
  \|\mathcal{M}f\|_{\ell^{1, \8}(\Z)}\lesssim \|f\|_{\ell^1(\Z)}.
\end{align}
\end{thm}
Before we prove Theorem \ref{thm:3CZ} we show how it implies Theorem \ref{thm:1max}. Indeed, it suffices to take
$K_n(x)=K_{h, 2^n}(x)$, $d_n\simeq\vp(2^n)$, $D_n\simeq2^n$ and
\begin{align*}
  F_n(x)=\left\{ \begin{array} {ll}
K_{h, 2^n}*\widetilde{K}_{h, 2^n}(x), & \mbox{if $0\le|x|\le\vp(2^n)$,}\\
\frac{1}{\vp(2^n)^2}\sum_{k\in\N}\vp'(k)\vp'(k+|x|)\eta\left(\frac {k}{2^n}\right)\eta\left(\frac{k+|x|}{2^n}\right),& \mbox{if $|x|>\vp(2^n)$.}
\end{array}
\right.
\end{align*}
It is easy to see that $F_n(x)$ has desired properties by Lemma \ref{lem:1ker0} and \ref{lem:1ker}.
\begin{proof}
Let $f\in\ell^1(\Z)$ and $\la>0$. We now perform a Calder\'{o}n--Zygmund decomposition at height $\la>0$. Then there exist a finite constant $C>0$, a set of indexes $\mathcal{B}\subseteq\N\cup\{0\}\times\Z$ and functions $g$ and $(b_{s, j})_{(s, j)\in\mathcal{B}}$ such that
\begin{align*}
  f=g+\sum_{(s, j)\in\mathcal{B}}b_{s, j}=g+\sum_{s\ge0}b_s,
\end{align*}
and
\begin{itemize}
  \item $\|g\|_{\ell^{\8}(\Z)}\le\la$,
  \item $b_{s, j}$ is supported on the dyadic cube $Q_{s, j}=[j2^s, (j+1)2^s)\cap\Z$,
  \item for any fixed $s\ge0$ $$b_s=\sum_{j\in\Z:\ (s, j)\in\mathcal{B}}b_{s, j},$$
  \item $\{Q_{s, j}: (s, j)\in\mathcal{B}\}$ is a disjoint collection,
  \item $\|b_{s, j}\|_{\ell^1(\Z)}\le\la|Q_{s, j}|=\la2^s$,
  \item The constant $C>0$ is independent of $\la>0$ and $f$. Moreover,  $$\sum_{(s, j)\in\mathcal{B}}|Q_{s, j}|\le\frac{C\|f\|_{\ell^1(\Z)}}{\la}.$$
\end{itemize}
Note that we have not assumed a cancellation condition for $b_{s, j}$. However, instead of that
we make further modifications of $b_{s, j}$. Namely,  we split $b_s$ as follows
\begin{align*}
  b_{s}=b_{s}^n+B_{s}^n+g_{s}^n,
\end{align*}
where (in the sequel we will use the following convenient notational convention $[f]_Q=\frac{1}{|Q|}\int_{Q}f$),
\begin{itemize}
  \item $b_{s}^n(x)=b_s(x)\mathbf{1}_{\{x\in\Z:\ |b_s(x)|>\la d_n\}}(x)$,
  \item $h_s^n(x)=b_s(x)-b_s^n(x)=b_s(x)\mathbf{1}_{\{x\in\Z:\ |b_s(x)|\le\la d_n\}}(x)$,
  \item $B_s^n(x)=h_s^n(x)-g_s^n(x)$, where
  \item for any fixed $s\ge0$  $$g_s^n(x)=\sum_{j\in\Z:\ (s, j)\in\mathcal{B}}[h_s^n]_{Q_{s, j}}\mathbf{1}_{Q_{s, j}}(x).$$
\end{itemize}
The task now is to show that $\mathcal{M}f(x)$ is of weak type $(1, 1)$. Since
\begin{align*}
  f=g+\sum_{s\ge0}g_{s}^n+\sum_{s\ge0}b_{s}^n+\sum_{s\ge0}B_{s}^n,
\end{align*}
we observe that
\begin{align*}
  \{x\in\Z:\ |\mathcal{M}f(x)|>4C\la\}&\subseteq\{x\in\Z:\ \sup_{n\in\N}\Big|K_n*\Big(g+\sum_{s\ge0}g_s^n\Big)(x)\Big|>C\la\}\\
  &\cup\{x\in\Z:\ \sup_{n\in\N}\Big|K_n*\Big(\sum_{s\ge0}b_s^n\Big)(x)\Big|>C\la\}\\
  &\cup\{x\in\Z:\ \sup_{n\in\N}\Big|K_n*\Big(\sum_{s=0}^{s(n)-1}B_s^n\Big)(x)\Big|>C\la\}\\
  &\cup\{x\in\Z:\ \sup_{n\in\N}\Big|K_n*\Big(\sum_{s=s(n)}^{\8}B_s^n\Big)(x)\Big|>C\la\}=S_1\cup S_2\cup S_3\cup S_4,
\end{align*}
 where $s(n)=\min\{s\in\N: 2^s\ge D_n\}$. We shall deal with each set separately.
\subsection{Step 1. Estimates for $|S_1|$.} If $C>0$ is sufficiently large then
\begin{align*}
  |S_1|\le|\{x\in\Z:\ \sup_{n\in\N}|K_n*g(x)|+\sup_{n\in\N}\Big|K_n*\Big(\sum_{s\ge0}g_s^n\Big)(x)\Big|>C\la\}|=0,
\end{align*}
since $\|g\|_{\ell^{\8}(\Z)}\le\la$ and $|[h_s^n]_{Q_{s, j}}|\le[|h_s^n|]_{Q_{s, j}}\le[|b_s^n|]_{Q_{s, j}}
\le [|b_s|]_{Q_{s, j}}\le|Q_{s, j}|^{-1}\|b_{s, j}\|_{\ell^1(\Z)}\le\la$, which in turn implies that
\begin{align*}
  \Big|\sum_{s\ge0}g_s^n(x)\Big|\le\sum_{(s, j)\in\mathcal{B}}|[h_s^n]_{Q_{s, j}}|\mathbf{1}_{Q_{s, j}}(x)
  \le\la\sum_{(s, j)\in\mathcal{B}}\mathbf{1}_{Q_{s, j}}(x)\le\la.
\end{align*}
\subsection{Step 2. Estimates for $|S_2|$.}
\begin{align*}
  |S_2|&\le|\{x\in\Z:\ \sup_{n\in\N}\Big|K_n*\Big(\sum_{s\ge0}b_s^n\Big)(x)\Big|>C\la\}|
  \le\sum_{n\in\N}\sum_{s\ge0}|\{x\in\Z:\ K_n*|b_s^n(x)|>0\}|\\
  &\le\sum_{n\in\N}\sum_{s\ge0}|\supp\ K_n|\cdot|\{x\in\Z:\ |b_s^n(x)|>0\}|\lesssim
  \sum_{n\in\N}\sum_{s\ge0}d_n\cdot|\{x\in\Z:\ |b_s(x)|>\la d_n\}|\\
  &\le\sum_{s\ge0}\sum_{n\in\N}d_n\sum_{k\ge n}|\{x\in\Z:\ \la d_k<|b_s(x)|\le\la d_{k+1}\}|\\
  &=\sum_{s\ge0}\sum_{k\in\N}|\{x\in\Z:\ \la d_k<|b_s(x)|\le\la d_{k+1}\}|
  \cdot\Big(\sum_{n=1}^kd_n\Big)\\
  &\lesssim\frac{1}{\la}\sum_{s\ge0}\sum_{k\in\N}\la d_k|\{x\in\Z:\ \la d_k<|b_s(x)|\le\la d_{k+1}\}|
  \lesssim \frac{1}{\la}\sum_{s\ge0}\|b_s\|_{\ell^1(\Z)}\lesssim \frac{\|f\|_{\ell^1(\Z)}}{\la},
\end{align*}
as desired.
\subsection{Step 3. Estimates for $|S_4|$.} It remains to show that
\begin{align*}
  |S_4|\le\frac{C\|f\|_{\ell^1(\Z)}}{\la}\ \ \mbox{for every $s\ge s(n),$}
\end{align*}
which will follow from the definition of $s(n)=\min\{s\in\N: 2^s\ge D_n\}$. Indeed, $\supp\ K_n\subseteq [0, D_n]\subseteq[0, 2^s]$ since  $s\ge s(n).$ Thus
\begin{align*}
  \supp\ K_n*B_s^n\subseteq\supp\ K_n+\supp\ B_s^n\subseteq[0, 2^s]+\bigcup_{k\in\Z}Q_{s, k}\subseteq\bigcup_{k\in\Z}3Q_{s, k},
\end{align*}
where $3Q$ denotes the unique cube with the same center as $Q$ and side length equal to $3$ times of the side length of $Q$. Therefore,
\begin{align*}
  S_4\subseteq\bigcup_{n\in\N}\bigcup_{s\ge s(n)}\{x\in\Z:\ |K_n*B_s^n(x)|>0\}\subseteq\bigcup_{n\in\N}\bigcup_{s\ge s(n)}\bigcup_{k\in\Z}3Q_{s, k}
  \subseteq\bigcup_{(s, k)\in\mathcal{B}}3Q_{s, k},
\end{align*}
and consequently
\begin{align*}
  |S_4|\le\sum_{(s, k)\in\mathcal{B}}|3Q_{s, k}|\lesssim\sum_{(s, k)\in\mathcal{B}}|Q_{s, k}|
  \le\frac{C\|f\|_{\ell^1(\Z)}}{\la}.
\end{align*}
\subsection{Step 4. Estimates for $|S_3|$.} What is left is to estimate $S_3$. For this purpose we will proceed as follows. Notice that by Lemma \ref{lem:4} we obtain
\begin{multline*}
  \la^2|\{x\in\Z:\ \sup_{n\in\N}\Big|K_n*\bigg(\sum_{s=0}^{s(n)-1}B_s^n\bigg)(x)\Big|>C\la\}|\\
  =\la^2|\{x\in\Z:\ \sup_{n\in\N}\Big|\sum_{s=0}^{s(n)-1}K_n*B_{s(n)-1-s}^n(x)\Big|>C\la\}|\\
  \lesssim\sum_{x\in\Z}\sup_{n\in\N}\Big|\sum_{s=0}^{s(n)-1}K_n*B_{s(n)-1-s}^n(x)\Big|^2
  \lesssim \sum_{n\in\N}\Big\|\sum_{s=0}^{s(n)-1}K_n*B_{s(n)-1-s}^n\Big\|_{\ell^2(\Z)}^2\\
  \le\sum_{n\in\N}\sum_{s=0}^{s(n)-1}\big\|K_n*B_{s(n)-1-s}^n\big\|_{\ell^2(\Z)}^2
  +2\sum_{n\in\N}\ \sum_{0\le s_2<s_1\le s(n)-1}\big|\big\langle K_n*B_{s(n)-1-s_1}^n, K_n*B_{s(n)-1-s_2}^n\big\rangle_{\ell^2(\Z)}\big|\\
  \lesssim\sum_{n\in\N}\sum_{s=0}^{s(n)-1}2^{-\d s}\la\|B_{s(n)-1-s}^n\|_{\ell^1(\Z)}
  +\sum_{n\in\N}\sum_{s=0}^{s(n)-1}d_n^{-1}\|B_{s(n)-1-s}^n\|_{\ell^2(\Z)}^2\\
  +\sum_{n\in\N}\ \sum_{0\le s_2<s_1\le s(n)-1}2^{-\d s_1}\la\|B_{s(n)-1-s_2}^n\|_{\ell^1(\Z)}.
\end{multline*}
 Then we can easily see that
  \begin{multline*}
    \sum_{n\in\N}\ \sum_{0\le s_2<s_1\le s(n)-1}2^{-\d s_1}\la\|B_{s(n)-1-s_2}^n\|_{\ell^1(\Z)}
  \lesssim\sum_{n\in\N}\ \sum_{s_2=0}^{s(n)-1}\ \sum_{s_1=s_2+1}^{s(n)-1}2^{-\d s_1}\la\|B_{s(n)-1-s_2}^n\|_{\ell^1(\Z)}\\
  \lesssim
  \sum_{n\in\N}\ \sum_{s_2=0}^{s(n)-1}2^{-\d s_2}\la\|B_{s(n)-1-s_2}^n\|_{\ell^1(\Z)}
  \lesssim \sum_{s_2\ge0}2^{-\d s_2}\la\ \sum_{s\in\N}\ \sum_{j\in\Z}\|b_{s, j}\|_{\ell^1(\Z)}\\
  \lesssim \la^2\sum_{(s, j)\in\mathcal{B}}|Q_{s, j}|\lesssim \la \|f\|_{\ell^1(\Z)}.
  \end{multline*}
  The proof will be completed if we show that
  \begin{align*}
  \sum_{n\in\N}\sum_{s=0}^{s(n)-1}d_n^{-1}\|B_{s}^n\|_{\ell^2(\Z)}^2\lesssim \la\|f\|_{\ell^1(\Z)}.
  \end{align*}
  For this purpose take $x\in Q_{s_0, j_0}$ and observe, since $B_s^n$'s have disjoint supports and $|[h_s^n]_{Q_{s, j}}|\le \la$, that
  \begin{multline*}
  \sum_{n\in\N}\sum_{s=0}^{s(n)-1}d_n^{-1}B_{s}^n(x)^2\lesssim
  \sum_{n\in\N}d_n^{-1}|b_{s_0}(x)|^2\mathbf{1}_{\{y\in\Z:\ |b_{s_0}(y)|\le\la d_n\}}(x)
  +\sum_{n\in\N}d_n^{-1}\la ^2\mathbf{1}_{\{\mbox{\tiny supp}\ b_{s_0}\}}(x)\\
  \lesssim |b_{s_0}(x)|^2\sum_{\genfrac{}{}{0pt}{}{n\in\N}{d_n\ge\la^{-1}|b_{s_0}(x)|}}d_n^{-1}+
  \la ^2\mathbf{1}_{\{\mbox{\tiny supp}\ b_{s_0}\}}(x)\lesssim \sum_{s\ge0}\la |b_{s}(x)|
  +\la ^2\mathbf{1}_{\{\mbox{\tiny supp}\ b_{s}\}}(x).
  \end{multline*}
  Therefore,
  \begin{multline*}
    \la^{-2} \sum_{n\in\N}\sum_{s=0}^{s(n)-1}d_n^{-1}\|B_{s}^n\|_{\ell^2(\Z)}^2
    \lesssim\frac{1}{\la}\sum_{s\ge0}\|b_s\|_{\ell^1(\Z)}+\sum_{x\in\Z}\mathbf{1}_{\{\mbox{\tiny supp}\ b_{s}\}}(x)\lesssim\sum_{(s, j)\in\mathcal{B}}|Q_{s, j}|\lesssim \frac{1}{\la}\|f\|_{\ell^1(\Z)},
  \end{multline*}
  as claimed, and the proof of Theorem \eqref{thm:3CZ} is finished.
\end{proof}

\begin{lem}\label{lem:4}
Under the assumptions of Theorem \ref{thm:3CZ} there exists $\d>0$ such that for every $0\le s_2<s_1\le s(n)-1$
\begin{align}\label{lem:4est1}
  \big|\big\langle K_n*B_{s(n)-1-s_1}^n, K_n*B_{s(n)-1-s_2}^n\big\rangle_{\ell^2(\Z)}\big|\lesssim
  2^{-\d s_1}\la\|B_{s(n)-1-s_2}^n\|_{\ell^1(\Z)},
\end{align}
and for every $0\le s\le s(n)-1$
\begin{align}\label{lem:4est2}
  \| K_n*B_{s(n)-1-s}^n\|_{\ell^2(\Z)}^2\lesssim
  2^{-\d s}\la\|B_{s(n)-1-s}^n\|_{\ell^1(\Z)}+d_n^{-1}\|B_{s(n)-1-s}^n\|_{\ell^2(\Z)}^2.
\end{align}
\end{lem}
\begin{proof}
According to \eqref{thm:3CZ2} we have
\begin{align*}
  K_n*\widetilde{K}_n(x)=F_n(0)\d_0(x)+G_n(x)+E_n(x),
\end{align*}
where $E_n(x)=K_n*\widetilde{K}_n(x)-F_n(x)$ and $G_n(x)=F_n(x)-F_n(0)\d_0(x)$. Moreover, $\supp\ G_n\subseteq[-D_n, D_n]\subseteq[-2^{r_n}, 2^{r_n}]$
and $\supp\ E_n\subseteq[-D_n, D_n]\subseteq[-2^{r_n}, 2^{r_n}]$, where $r_n=\log_2(\lfloor D_n\rfloor+1)$. Therefore,  taking $Z_{j, n}=[j2^{r_n}, (j+1)2^{r_n})$, in view of \eqref{thm:3CZ2} and \eqref{thm:3CZ3}, we obtain for every $0\le s_2\le s_1\le s(n)-1$ that
\begin{multline}\label{lem:4p1}
  \big|\big\langle K_n*B_{s(n)-1-s_1}^n, K_n*B_{s(n)-1-s_2}^n\big\rangle_{\ell^2(\Z)}\big|
  =\big|\big\langle K_n*\widetilde{K}_n*B_{s(n)-1-s_1}^n, B_{s(n)-1-s_2}^n\big\rangle_{\ell^2(\Z)}\big|\\
  \lesssim
  F_n(0)\big|\big\langle \d_0*B_{s(n)-1-s_1}^n, B_{s(n)-1-s_2}^n\big\rangle_{\ell^2(\Z)}\big|
  +\big|\big\langle (G_n+E_n)*B_{s(n)-1-s_1}^n, B_{s(n)-1-s_2}^n\big\rangle_{\ell^2(\Z)}\big|\\
  \lesssim d_n^{-1}\big|\big\langle B_{s(n)-1-s_1}^n, B_{s(n)-1-s_2}^n\big\rangle_{\ell^2(\Z)}\big|\\
  +\sum_{j\in\Z}\sum_{x\in Z_{j, n}}\Big|\sum_{y\in\Z}G_n(y)B_{s(n)-1-s_1}^n(x-y)\mathbf{1}_{[(j-1)2^{r_n}, (j+2)2^{r_n})}(x-y)\Big|
  |B_{s(n)-1-s_2}^n(x)|\\
  +\sum_{j\in\Z}\sum_{x\in Z_{j, n}}\sum_{y\in\Z}|E_n(y)B_{s(n)-1-s_1}^n(x-y)\mathbf{1}_{[(j-1)2^{r_n}, (j+2)2^{r_n})}(x-y)||B_{s(n)-1-s_2}^n(x)|\\
  \lesssim d_n^{-1}\big|\big\langle B_{s(n)-1-s_1}^n, B_{s(n)-1-s_2}^n\big\rangle_{\ell^2(\Z)}\big|\\
  +\sup_{j\in\Z}\sup_{x\in Z_{j, n}}\big|G_n*\big(B_{s(n)-1-s_1}^n\mathbf{1}_{[(j-1)2^{r_n}, (j+2)2^{r_n})}\big)(x)\big|\|B_{s(n)-1-s_2}^n\|_{\ell^1(\Z)}\\
  +D_n^{-1-\e_1}\sup_{j\in\Z}\|B_{s(n)-1-s_1}^n\mathbf{1}_{[(j-1)2^{r_n}, (j+2)2^{r_n})}\|_{\ell^1(\Z)}\|B_{s(n)-1-s_2}^n\|_{\ell^1(\Z)}=I_1+I_2+I_3.
\end{multline}
Now it is easy to see that
\begin{align*}
  \big\langle B_{s(n)-1-s_1}^n, B_{s(n)-1-s_2}^n\big\rangle_{\ell^2(\Z)}=\left\{ \begin{array} {ll}
\|B_{s(n)-1-s}^n\|_{\ell^2(\Z)}^2, & \mbox{if $s=s_1=s_2$,}\\
\ \ \ \ \ \ \ \  0,& \mbox{if $s_1\not=s_2$,}
\end{array}
\right.
\end{align*}
since the supports of $B_{s(n)-1-s_1}^n, B_{s(n)-1-s_2}$ are disjoint for $s_1\not=s_2$. Therefore, it remains to estimate the last two summands $I_2, I_3$ in \eqref{lem:4p1}. In order to find an  upper bound for $I_2$ we will use the fact that $G_n(x)$ is slowly varying away from $0$ (see \eqref{thm:3CZ3} and \eqref{thm:3CZ4}). An upper bound for $I_3$ follows from the definition of $E_n(x)$ and \eqref{thm:3CZ2}.  Define $B_{s}^{n, k}=B_{s}^{n}\mathbf{1}_{Q_{s, k}}$ for every $k\in\Z$ and observe that $\sum_{x\in\Z}B_{s}^{n, k}(x)=0$ and
\begin{align}\label{lem:4p0}
  \|B_{s}^{n, k}\|_{\ell^1(\Z)}\le\|h_s^n\mathbf{1}_{Q_{s, k}}\|_{\ell^1(\Z)}
  +\|[h_s^n]_{Q_{s, k}}\mathbf{1}_{Q_{s, k}}\|_{\ell^1(\Z)}
  \lesssim \|b_{s, k}\|_{\ell^1(\Z)}\le\la|Q_{s, k}|.
\end{align}
This in turn implies that for every $\widetilde{C}>0$ and $0<a_1<a_2$ we have
\begin{multline}\label{lem:4p2}
  \|B_{s}^{n}\mathbf{1}_{[(\widetilde{C}k-a_1)2^{r_n}, (\widetilde{C}k+a_2)2^{r_n})}\|_{\ell^1(\Z)}\le
  \sum_{k\in\Z:\ Q_{s, k}\cap[(\widetilde{C}k-a_1)2^{r_n}, (\widetilde{C}k+a_2)2^{r_n})\not=\emptyset}\|B_{s}^{n, k}\|_{\ell^1(\Z)}\\
  \lesssim |\{k\in\Z:\ Q_{s, k}\cap[(\widetilde{C}k-a_1)2^{r_n}, (\widetilde{C}k+a_2)2^{r_n})\not=\emptyset\}|\cdot\la2^s
  \lesssim (a_2-a_1)2^{r_n}2^{-s}\la2^s\lesssim \la D_n.
\end{multline}
This  gives immediately an upper bound for $I_3$
\begin{align*}
  I_3&=D_n^{-1-\e_1}\sup_{j\in\Z}\|B_{s(n)-1-s_1}^n\mathbf{1}_{[(j-1)2^{r_n}, (j+2)2^{r_n})}\|_{\ell^1(\Z)}\|B_{s(n)-1-s_2}^n\|_{\ell^1(\Z)}\\
  &\lesssim
  D_n^{-1-\e_1}\la D_n\|B_{s(n)-1-s_2}^n\|_{\ell^1(\Z)}
  \lesssim 2^{-\e_1 s_1}\la\|B_{s(n)-1-s_2}^n\|_{\ell^1(\Z)},
\end{align*}
since $D_n^{-1}<2^{-s_1}$ due to the inequality $s_1\le s(n)-1$. The proof will be completed if we find an upper bound for $I_2$. First of all notice that
\begin{multline}\label{lem:4p3}
  \sup_{j\in\Z}\sup_{x\in Z_{j, n}}\big|G_n*\big(B_{s(n)-1-s_1}^n\mathbf{1}_{[(j-1)2^{r_n}, (j+2)2^{r_n})}\big)(x)\big|\\
  \le\sup_{j\in\Z}\sup_{x\in Z_{j, n}}\sum_{k\in\Z}\big|G_n*\big(B_{s(n)-1-s_1}^{n, k}\mathbf{1}_{[(j-1)2^{r_n}, (j+2)2^{r_n})}\big)(x)\big|.
\end{multline}
Furthermore,
\begin{align}\label{lem:4p4}
\sum_{x\in\Z}B_{s(n)-1-s_1}^{n, k}(x)\mathbf{1}_{[(j-1)2^{r_n}, (j+2)2^{r_n})}(x)=0.
\end{align}
This is trivial if $\supp\ B_{s(n)-1-s_1}^{n, k}\cap[(j-1)2^{r_n}, (j+2)2^{r_n})=\emptyset$. Consider, now the case when $Q_{s(n)-1-s_1, k}\cap[(j-1)2^{r_n}, (j+2)2^{r_n})\not=\emptyset$, then $Q_{s(n)-1-s_1, k}$ is contained in $[(j-1)2^{r_n}, (j+2)2^{r_n})$, since $s(n)-1-s_1\le r_n$ and the last interval is the sum of three dyadic sets of length $2^{r_n}$. Thus
\begin{multline*}
 \sum_{x\in\Z}B_{s(n)-1-s_1}^{n, k}(x)\mathbf{1}_{[(j-1)2^{r_n}, (j+1)2^{r_n})}(x)\\
 =\sum_{x\in\Z}h^n_{s(n)-1-s_1}\mathbf{1}_{Q_{s(n)-1-s_1, k}}(x)
 -[h_{s(n)-1-s_1}^n]_{Q_{s(n)-1-s_1, k}}\mathbf{1}_{Q_{s(n)-1-s_1, k}}(x)=0.
\end{multline*}
Fix $k, j\in\Z$ and let $x_{s(n)-1-s_1, k}$ be the center of the cube $Q_{s(n)-1-s_1, k}$ and take any
$x\in Z_{j, n}$ such that $|x-x_{s(n)-1-s_1, k}|\ge Cd_n^{\e_2}+C2^{s(n)-1-s_1}$ then using \eqref{lem:4p4} and \eqref{thm:3CZ4} we see
\begin{multline}\label{lem:4p5}
 \big|G_n*\big(B_{s(n)-1-s_1}^{n, k}\mathbf{1}_{[(j-1)2^{r_n}, (j+2)2^{r_n})}\big)(x)\big|\\
 =\Big|\sum_{y\in\Z}\big(G_n(x-y)-G_n(x-x_{s(n)-1-s_1,k})\big)B_{s(n)-1-s_1}^{n, k}(y)\mathbf{1}_{[(j-1)2^{r_n}, (j+2)2^{r_n})}(y)\Big|\\
 \lesssim\sum_{y\in\Z}\frac{|x_{s(n)-1-s_1, k}-y|}{D_n^2}|B_{s(n)-1-s_1}^{n, k}(y)\mathbf{1}_{[(j-1)2^{r_n}, (j+2)2^{r_n})}(y)|\\
 \lesssim
 \frac{2^{s(n)-1-s_1}}{D_n^2}\|B_{s(n)-1-s_1}^{n, k}\mathbf{1}_{[(j-1)2^{r_n}, (j+2)2^{r_n})}\|_{\ell^1(\Z)},
\end{multline}
since $|x_{s(n)-1-s_1, k}-y|\le 2^{s(n)-1-s_1}$ and
\begin{align*}
  |x-y|\ge |x-x_{s(n)-1-s_1, k}|-|x_{s(n)-1-s_1, k}-y|\ge Cd_n^{\e_2}+C2^{s(n)-1-s_1}-2^{s(n)-1-s_1}\gtrsim d_n^{\e_2}.
\end{align*}
On the other hand in view of \eqref{thm:3CZ3} we have for all $x\in\Z\setminus\{0\}$
\begin{multline}\label{lem:4p6}
 \big|G_n*\big(B_{s(n)-1-s_1}^{n, k}\mathbf{1}_{[(j-1)2^{r_n}, (j+2)2^{r_n})}\big)(x)\big|
 \lesssim D_n^{-1}\|B_{s(n)-1-s_1}^{n, k}\mathbf{1}_{[(j-1)2^{r_n}, (j+2)2^{r_n})}\|_{\ell^1(\Z)}.
\end{multline}
Now we can continue with estimating \eqref{lem:4p3}. Indeed, by \eqref{lem:4p5} and \eqref{lem:4p6}
\begin{multline}\label{lem:4p7}
  \sup_{j\in\Z}\sup_{x\in Z_{j, n}}\sum_{k\in\Z}\big|G_n*\big(B_{s(n)-1-s_1}^{n, k}\mathbf{1}_{[(j-1)2^{r_n}, (j+2)2^{r_n})}\big)(x)\big|\\
  \lesssim\sup_{j\in\Z}\sup_{x\in Z_{j, n}}\sum_{k\in\Z:\ |x-x_{s(n)-1-s_1, k}|< Cd_n^{\e_2}+C2^{s(n)-1-s_1}}
  D_n^{-1}\|B_{s(n)-1-s_1}^{n, k}\mathbf{1}_{[(j-1)2^{r_n}, (j+2)2^{r_n})}\|_{\ell^1(\Z)}\\
  +\sup_{j\in\Z}\sup_{x\in Z_{j, n}}\sum_{k\in\Z:\ |x-x_{s(n)-1-s_1, k}|\ge Cd_n^{\e_2}+C2^{s(n)-1-s_1}}
  \frac{2^{s(n)-1-s_1}}{D_n^2}\|B_{s(n)-1-s_1}^{n, k}\mathbf{1}_{[(j-1)2^{r_n}, (j+2)2^{r_n})}\|_{\ell^1(\Z)}.
\end{multline}
In order to estimate the first sum we need to consider two cases. Firstly, assume that $2^{s(n)-1-s_1}\le d_n^{\e_2}$, then any ball with radius $\lesssim d_n^{\e_2}$ contains at most $d_n^{\e_2}2^{-(s(n)-1-s_1)}$
cubes of the form $Q_{s(n)-1-s_1, k}$. Thus by \eqref{lem:4p0} we obtain
\begin{align*}
  \sup_{j\in\Z}\sup_{x\in Z_{j, n}}\sum_{k\in\Z:\ |x-x_{s(n)-1-s_1, k}|< Cd_n^{\e_2}+C2^{s(n)-1-s_1}}
  D_n^{-1}\|B_{s(n)-1-s_1}^{n, k}\mathbf{1}_{[(j-1)2^{r_n}, (j+2)2^{r_n})}\|_{\ell^1(\Z)}\\
  \lesssim \sup_{j\in\Z}\sup_{x\in Z_{j, n}}|\{k\in\Z:\ |x-x_{s(n)-1-s_1, k}|< Cd_n^{\e_2}\}|
  D_n^{-1}\la 2^{s(n)-1-s_1}\\
  \lesssim d_n^{\e_2}2^{-(s(n)-1-s_1)}D_n^{-1}\la 2^{s(n)-1-s_1}\le
  \frac{\la d_n}{D_n^{\e_0}D_n^{1-\e_0}}\lesssim\frac{\la}{2^{(1-\e_0)s_1}}.
\end{align*}
Secondly, assume that $d_n^{\e_2}\le 2^{s(n)-1-s_1}$, then any ball with radius $\lesssim 2^{s(n)-1-s_1}$ contains at most $C2^{s(n)-1-s_1}2^{-(s(n)-1-s_1)}=C$
cubes of the form $Q_{s(n)-1-s_1, k}$. Thus by \eqref{lem:4p0} we obtain
\begin{align*}
  \sup_{j\in\Z}\sup_{x\in Z_{j, n}}\sum_{k\in\Z:\ |x-x_{s(n)-1-s_1, k}|< Cd_n^{\e_2}+C2^{s(n)-1-s_1}}
  D_n^{-1}\|B_{s(n)-1-s_1}^{n, k}\mathbf{1}_{[(j-1)2^{r_n}, (j+2)2^{r_n})}\|_{\ell^1(\Z)}\\
  \lesssim \sup_{j\in\Z}\sup_{x\in Z_{j, n}}|\{k\in\Z:\ |x-x_{s(n)-1-s_1, k}|< C2^{s(n)-1-s_1}\}|
  D_n^{-1}\la 2^{s(n)-1-s_1}\\
  \lesssim D_n^{-1}\la 2^{s(n)-1-s_1}\lesssim\la 2^{-s_1}.
\end{align*}
Using \eqref{lem:4p2} we can easily estimate the second sum in \eqref{lem:4p7}. Namely,
\begin{align*}
 \sup_{j\in\Z}\sup_{x\in Z_{j, n}}\sum_{k\in\Z:\ |x-x_{s(n)-1-s_1, k}|\ge Cd_n^{\e_2}+C2^{s(n)-1-s_1}}
  \frac{2^{s(n)-1-s_1}}{D_n^2}\|B_{s(n)-1-s_1}^{n, k}\mathbf{1}_{[(j-1)2^{r_n}, (j+2)2^{r_n})}\|_{\ell^1(\Z)}\\
  \lesssim \frac{2^{-s_1}}{D_n}\sup_{j\in\Z}\sum_{k\in\Z}
  \|B_{s(n)-1-s_1}^{n}\mathbf{1}_{Q_{s(n)-1-s_1,k}}\mathbf{1}_{[(j-1)2^{r_n}, (j+2)2^{r_n})}\|_{\ell^1(\Z)}\\
  \lesssim \frac{2^{-s_1}}{D_n}\sup_{j\in\Z}
  \|B_{s(n)-1-s_1}^{n}\mathbf{1}_{[(j-1)2^{r_n}, (j+2)2^{r_n})}\|_{\ell^1(\Z)}
  \lesssim\frac{2^{-s_1}}{D_n}\la D_n=2^{-s_1}\la.
\end{align*}
Finally, we obtain the upper bound for $I_2$
\begin{align*}
  I_2=\sup_{j\in\Z}\sup_{x\in Z_{j, n}}\big|G_n*\big(B_{s(n)-1-s_1}^n\mathbf{1}_{[(j-1)2^{r_n}, (j+2)2^{r_n})}\big)(x)\big|\|B_{s(n)-1-s_2}^n\|_{\ell^1(\Z)}\\
  \lesssim2^{-\d s_1}\la\|B_{s(n)-1-s_2}^n\|_{\ell^1(\Z)},
\end{align*}
for some $\d>0$ and the proof of Lemma \ref{lem:4} is completed.
\end{proof}

\end{document}